\documentclass[letterpaper,11pt,reqno]{amsart}

\makeatletter
\usepackage{amssymb}
\usepackage{latexsym}
\usepackage{amsbsy}
\usepackage{amsfonts}
\usepackage{hyperref}
\usepackage{graphicx}
\usepackage{enumerate}
\usepackage{mathtools}
\usepackage{color}
\usepackage{comment}
\usepackage{ulem}

 \newcommand{\timed}{\partial_t^\tau}%discrete deriv wrt time
\usepackage{tikz}

\def\marginpar#1{\ignorespaces}

\DeclareMathOperator{\Tr}{Tr}

  \newcommand{\beq}{\begin{equation}}
    \newcommand{\eeq}{\end{equation}}
    
    \newcommand{\bal}{\begin{align}}
    \newcommand{\eal}{\end{align}}
    \newcommand{\bals}{\begin{align*}}
    \newcommand{\eals}{\end{align*}}
    
    \newcommand{\bbN}{{\mathbb{N}}}
    \newcommand{\bbR}{{\mathbb{R}}}

    \newcommand{\bbZ}{{\mathbb{Z}}}

    \newcommand{\calE}{{\mathcal E}}

    \newcommand{\calL}{{\mathcal L}}
    
    \newcommand{\calF}{{\mathcal F}}

    \newcommand{\eps}{\varepsilon}
    
    \newcommand{\lb}{\label}

\DeclareMathOperator\Lip{Lip}

\DeclareMathOperator\argmin{\arg \min}
\DeclareMathOperator\argmax{\arg \max}

\newtheorem{theorem}{Theorem}[section]
\newtheorem{remark}{Remark}[section]
\newtheorem{ex}{Example}[section]
\newtheorem{lemma}[theorem]{Lemma}
\newtheorem{proposition}[theorem]{Proposition}
\newtheorem{corollary}[theorem]{Corollary}
\newtheorem{definition}[theorem]{Definition}

\numberwithin{equation}{section}
\makeatother
\begin{document}
\title[Policy iteration for nonconvex viscous HJ equations]{Policy iteration for nonconvex viscous Hamilton--Jacobi equations}

\author[X. Guo, H. V. Tran, Y. P. Zhang]{Xiaoqin Guo, Hung Vinh Tran, Yuming Paul Zhang}

\thanks{
%The work of HT is partially supported by NSF grant DMS-2348305. 
X. Guo is supported by Simons Foundation through Collaboration Grant for Mathematicians \#852943. 
H. V. Tran is partially supported by NSF grant DMS-2348305. Y. P. Zhang is partially supported by NSF CAREER grant DMS-2440215, Simons Foundation Travel Support MPS-TSM-00007305, and a start-up grant at Auburn University.
}

\address[X. Guo]
{
Department of Mathematical Sciences, 
University of Cincinnati,
2815 Commons Way, Cincinnati OH 45221-0025}
\email{guoxq@ucmail.uc.edu}

\address[H. V. Tran]
{
Department of Mathematics, 
University of Wisconsin Madison, Van Vleck Hall, 480 Lincoln Drive, Madison, WI 53706}
\email{hung@math.wisc.edu}

\address[Y. P. Zhang]
{
Department of Mathematics and Statistics, Auburn University, Parker Hall, 
221 Roosevelt Concourse, Auburn, AL 36849}
\email{yzhangpaul@auburn.edu}

\date{\today} 


\begin{abstract}
We study the convergence rates of policy iteration (PI) for nonconvex viscous Hamilton--Jacobi equations using a discrete space-time scheme, where both space and time variables are discretized. 
We analyze the case with an uncontrolled diffusion term, which corresponds to a possibly degenerate viscous Hamilton--Jacobi equation. 
 We first obtain an exponential convergent result of PI for the discrete space-time schemes.
We then investigate the discretization error.
\end{abstract}

\maketitle

%\setcounter{tocdepth}{1}
%\tableofcontents

%-------------------------------------------------------------------------------------------------
\section{Introduction}
%\label{sc1}

\subsection{Settings} 
In this paper, we are interested in the policy iteration (PI) for the following viscous Hamilton--Jacobi equation
\begin{equation}
\label{3.1}
\begin{cases}
\partial_t v(t,x) + H(t,x, \nabla v(t,x)) =-\frac12\Tr ((\sigma\sigma^T)(t,x) D^2 v(t,x)) \qquad &\text{ in } (0,T)\times \mathbb{R}^d,\\
 v(T,x) = g(x) \qquad &\text{ on }  \mathbb{R}^d.
 \end{cases}
\end{equation}

\quad We begin by providing a motivation for studying the nonconvex viscous Hamilton--Jacobi equation \eqref{3.1}.
 We consider a zero-sum diﬀerential game played by two
players, I and II, who are both rational. 
In the game, Player I aims to minimize while Player II aims to maximize a certain payoff functional by controlling the dynamics of the system state, which represents the location of the pair in the game.
Fix $T>0$ and $d, m_\alpha, m_\beta \in \bbN$.
Let $A \subseteq \mathbb{R}^{m_\alpha}$ and $ B \subseteq \mathbb{R}^{m_\beta}$ be compact sets.
For $(t,x)\in (0,T)\times \bbR^d$, the system state is governed by the stochastic differential equation:
\begin{equation*}%\lb{2.1}
\begin{cases}
d X(s) = f(s,X(s), a(s),b(s))\,ds+\sigma(s,X(s))\,d B_s \quad \text{ for } s\in (t,T],\\
X(t)=x\in \bbR^d.
\end{cases}
\end{equation*}
Here, $X(s) \in \mathbb{R}^d$ is the system state, and $B_s$ denotes the \( d \)-dimensional Brownian motion. 
Let \(\mathcal{F} = (\mathcal{F}_s)_{s \geq 0}\) be the filtration generated by \((B_s)_{s \geq 0}\). 
We assume that $a:[t,T] \to A$, $b:[t,T]\to B$ are $\{\calF_s\}_{s\in [t,T]}$-adapted processes (or controls or policies). 
The standard admissible controls for Players I and II in time $[t,T]$ are denoted by $\mathcal{A}_t$ and $ \mathcal{B}_t$, respectively, where
\begin{align*}
&\mathcal{A}_t= \{a:[t,T] \to A \,:\, a \text{ is a $\{\calF_s\}_{s\in [t,T]}$-adapted process} \},\\
&\mathcal{B}_t= \{b:[t,T] \to B \,:\, b \text{ is a $\{\calF_s\}_{s\in [t,T]}$-adapted process} \}.
\end{align*}
We identify any two controls which agree a.e.
%$a\in \mathcal{A}_t, b\in \mathcal{B}_t$, 
Assume that $f:[0,T]\times \bbR^d\times A\times B\to\bbR^d$ and $\sigma:[0,T]\times \bbR^d\to\bbR^{d\times d}$ are uniformly bounded and Lipschitz continuous, and $\sigma$ is diagonal, that is, $\sigma={\rm diag}(\sigma_1,\ldots,\sigma_d)$.
We denote $\Sigma_i=\sigma_i^2$ for $1\leq i\leq d$ and $\Sigma={\rm diag}(\Sigma_1,\ldots,\Sigma_d)$.
Note that $\sigma, \Sigma$ do not depend on the controls in our setting.

\quad Let $c:[0,T]\times \bbR^d\times A\times B\to\bbR$ and $g:\bbR^d \to \bbR$ be Lipschitz continuous functions, which represent the running cost and the terminal cost, respectively.
The upper value function of the game is given by
\begin{equation}\lb{2.2}
v_*(t,x):= \sup_{\beta\in\Gamma_t} \inf_{a \in \mathcal{A}_t} E_{tx}\left[\int_t^T c(s,X(s), a(s),\beta[a](s))\,ds + g(X(T)) \right],
\end{equation}
Here, the set of strategies for player II beginning at time $t$ is
\[
\Gamma_t=\{\beta:\mathcal{A}_t \to \mathcal{B}_t \text{ nonanticipating}\},
\]
where nonanticipating means that, for $a_1,a_2\in \mathcal{A}_t$ and $s\in [t,T]$,
\[
a_1(\cdot)=a_2(\cdot) \text{ on } [t,s) \quad \implies \quad \beta[a_1](\cdot)=\beta[a_2](\cdot) \text{ on } [t,s).
\]
And $E_{tx}[G]$ denotes the expected value of $G$.
It is known that under suitable assumptions (see \cite{Fleming-Souganidis}), 
$v_*$ defined by \eqref{2.2} is the viscosity solution to \eqref{3.1}, where  the Hamiltonian $H:[0,T]\times\bbR^d\times\bbR^d\times \bbR^{d\times d}\to \bbR$ is given by
\beq\lb{3.1'}
H(t,x,p):= \sup_{b\in B}\inf_{a \in A} L(t,x,p)(a,b)
\eeq
and
\[
L(t,x,p)(a,b):=c(t,x,a,b) + p \cdot f(t,x,a,b).
\]
In general, the second-order term in \eqref{3.1} could be degenerate as we only have 
\[
\Sigma=\sigma \sigma^T=\sigma^2 ={\rm diag}(\Sigma_1,\ldots,\Sigma_d) \geq 0.
\]
We say that $\Sigma$ is the diffusion matrix.
When $\sigma=0$, \eqref{3.1} becomes a first-order Hamilton--Jacobi equation (see \cite[Chapter 3]{tranbook}).
In the literature, \eqref{3.1} is also called a Bellman--Isaacs equation.
We note that we do not use the game-theoretic framework in the analysis of this paper.

\quad Let us assume that the infimum and supremum in \eqref{3.1'} can be achieved. 
More precisely, we assume throughout this paper that the following $\alpha_b$ and $\beta$ are always well-defined.
For each $b\in B$, denote by
\begin{equation}
\label{eq:a}
\alpha_b(t,x,p)\in \argmin_{a \in A} L(t,x,p)(a,b) 
\end{equation}
and, with one fixed selection $\alpha_b$,
\begin{equation}
\label{eq:b}
\beta(t,x,p)\in \argmax_{b \in B}  L(t,x,p)(\alpha_b(t,x,p),b) .
\end{equation}
Actually, we will only need $\alpha_b$ and $\beta$ to be well-defined for $(t,x)$ in discrete space-time grids.
If $v_*$ is smooth enough, the optimal policy corresponding to \eqref{2.2} is 
\begin{equation*}
\beta_{*}(t,x) = \beta(t,x, \nabla v_{*}(t,x))
\end{equation*}
and
\[
\alpha_{*}(t,x) = \alpha_{\beta_*(t,x)}(t,x, \nabla v_{*}(t,x)).
\]
Of course, if $v_*$ is only Lipschitz, $\alpha_*$ and $\beta_*$ are not well-defined in the classical sense.

\quad Policy iteration is an approximate dynamic programming, 
which alternates between policy evaluation to obtain the value function with the current control
and policy improvement to optimize the value function.
More precisely, for $n = 0,1, \cdots$, the iterative procedure is as follows:
\begin{itemize}
\item
Given $\alpha_n=\alpha_n(t,x),\beta_n=\beta_n(t,x)$, solve the linear PDE 
%(we drop the notation of $(t,x)$)
\begin{equation}
\label{eq:PDEiter}
\begin{cases}
\partial_t v_n + L(t,x,\nabla v_n)(\alpha_n,\beta_n) = -\frac12\Tr ((\sigma\sigma^T)(t,x) D^2 v_n)  &\text{ in } (0,T)\times \mathbb{R}^d,\\
v_n(T,x) =g(x)  &\text{ on }  \mathbb{R}^d.
 \end{cases}
 \end{equation}
\item
Set
\begin{equation}
\label{eq:uiter}
\begin{aligned}
\alpha_{n+1,b}(t,x)&\in \argmin_{a \in A} L(t,x,\nabla v_n(t,x))(a,b),\\
\beta_{n+1}(t,x)&\in \argmax_{b \in B} L(t,x,\nabla v_n(t,x))(\alpha_{n+1,b}(t,x),b),\\
\alpha_{n+1}(t,x) &= \alpha_{n+1,\beta_{n+1}(t,x)}(t,x).    
\end{aligned}
\end{equation}
\end{itemize}
Thanks to our definition,
\[
H(t,x,\nabla v_n)=L(t,x,\nabla v_n)(\alpha_{n+1},\beta_{n+1}).
\]
Besides, we note the following important inequalities, for any $a\in A, b\in B$,
\begin{equation*}%\label{eq:crucial ine}
    L(t,x,\nabla v_n)(\alpha_{n+1,b},b)
    \leq L(t,x,\nabla v_n)(\alpha_{n+1},\beta_{n+1})
    \leq L(t,x,\nabla v_n)(a,\beta_{n+1}).
\end{equation*}
A key question about PI is to understand how the sequence $\{v_n\}$ approximates the optimal value $v_{*}$, and how the sequence of policies $\{(\alpha_n,\beta_n)\}$ approximates the optimal policy $(\alpha_*,\beta_*)$. 
We note that $(\alpha_n,\beta_n)$ is not necessarily unique for each $n\in \bbN$, and $(\alpha_*,\beta_*)$ might not be well-defined in the classical sense.
Furthermore, we do not have any information on the regularity of $\alpha_n,\beta_n$ with respect to $(t,x)$ as we do not place any such assumption in \eqref{eq:a}--\eqref{eq:b}.
These points contribute to making the problem both extremely challenging and interesting.
To the best of our knowledge, there was no result in the literature regarding PI for nonconvex Hamilton--Jacobi equations.

\quad We study the PI using a discrete space-time scheme, where both the space and time variables are discretized. 
In a given discrete space-time grid, \eqref{eq:a}--\eqref{eq:b} are well-defined and we do not need to worry about the measurability of $\alpha_n,\beta_n$ with respect to $(t,x)$.

%%\quad We write $C$ as various universal constants that only depend on $d$ and the constants in (A1)--(A2) unless otherwise stated. Specifically, since $T,h$ are not universal constants, we keep track of the dependence on $T,h$ in most estimates. 
%The constants $C$ might vary from one line to another. By $C_X$ or $C(X)$ we mean a constant that depends on universal constants and $X$. %{(in particular, $X$ can be $T$ or $h$)}.}

%---------------------------------------------
\subsection{Discrete space-time schemes}
We start with the notations. 
Denote by $\bbN$ the set of all natural numbers, and $\bbZ$ as the set of all integers. 
For any $h>0$, we write $\bbZ^d_h=h\bbZ^d:=\{hz\,|\,z\in\bbZ^d\}$. 
Let $\bbR^d$ be the Euclidean space of dimension $d$ and $|\cdot|$ the Euclidean distance. 
Denote by $\mathcal{S}^d$ the set of all symmetric matrices of size $d\times d$.
For $A\in \mathcal{S}^d$, $\Tr(A)$ denotes the trace of matrix $A$. 
For $R>0$, by $B_R(x)$ and $\overline B_R(x)$ we mean the open ball and close ball  in $\bbR^d$ with center $x\in \bbR^d$ and radius $R$, respectively. 
We write $B_R=B_R(0)$ and $\overline B_R = \overline B_R(0)$.
For a vector field $f:\Omega\to \bbR^d$ where $\Omega\subset\bbR^l$ for some $l\geq 1$, we denote its infinity norm by $\|f\|_\infty:=\sup_\Omega |f(\cdot)|$. 
For a function $g:[0,T]\times \bbR^d\to \bbR$, the spatial gradient, the spatial Hessian are denoted as $\nabla g(t,x)=\nabla_x\, g(t,x)$, $D^2 g(t,x)= D^2_{xx} g(t,x)$, respectively, and the partial derivative with respect to time is denoted as $\partial_t g(t,x)$. 
If $g$ is bounded and uniformly Lipschitz continuous,  we denote by $\|g\|_{\rm Lip}$ the sum of $\|g\|_\infty$ and the Lipschitz constant of $g$.

\quad In this paper, we always assume that
\beq\lb{N1}
c(t,x,\alpha,\beta),\,f(t,x,\alpha,\beta),\,\sigma(t,x),\,g(x) \quad \text{ are uniformly bounded.} 
\eeq 
We consider the discrete scheme in space and time.
Let ${\tau,h}\in (0,1)$, and assume that $T/\tau\in\bbN$. Denote 
\[
{\bbN}^\tau_{T}:=\{0,\tau,2\tau,\cdots,T\},\quad \bbZ^d_h:=h\bbZ^d,\quad\text{and}\quad \Omega^{\tau,h}_{T}:={\bbN}^\tau_{T}\times \bbZ^d_h.
\]
%\[
%\Omega^{\tau,h}_{T}:={\bbN}^\tau_{T}\times \bbZ^d_h\quad\text{and}\quad \Omega^{\tau,h}_{T}:=({\bbN}^\tau_{T}\backslash\{T\})\times \bbZ^d_h.
%\]

\quad For any $\varphi:{\bbN}^\tau_{T}\times \bbR^d\to\bbR$ and $h\in\bbR \setminus \{0\}$, we use the notations
\[
\nabla^h \varphi(t,x):=\left(\frac{\varphi(t,x+he_1)-\varphi(t,x-he_1)}{2h},\cdots,\frac{\varphi(t,x+he_d)-\varphi(t,x-he_d)}{2h}\right),
\]
\[
D^h \varphi(t,x):=\left(D^h_i\varphi(t,x)\right)_{1\leq i\leq d}=\left(\frac{\varphi(t,x+he_1)-\varphi(t,x)}{h},\cdots,\frac{\varphi(t,x+he_d)-\varphi(t,x)}{h}\right),
\]
%\[
%\Delta^h \varphi(t,x):=\sum_{i=1}^d\frac{\varphi(t,x+he_i)-2\varphi(t,x)+\varphi(t,x-he_i)}{h^2},
%\]
and
\[%\beq\lb{Dh2}
\Delta^{h}_i\varphi(t,x):=\frac{\varphi(t,x+he_i)-2\varphi(t,x)+\varphi(t,x-he_i)}{h^2}.
\]
It is not hard to see that
\beq\lb{p.9}
\Delta^{h}_i\varphi(t,x)=-D_i^{-h}D_i^h\varphi=\frac{1}{h}(D_i^h\varphi+D_{-i}^h\varphi),%\quad\Delta^h\varphi(t,x)=\sum_{i=1}^d\Delta^{h}_i\varphi(t,x),
\eeq
\begin{equation*}%\lb{p.6}
\nabla^h \varphi(t,x)=\frac{1}{2}\left(D^h \varphi(t,x)-D^{-h} \varphi(t,x)\right).%=\Tr \left(D^{h,2}\varphi(t,x)\right).
\end{equation*}
We also denote
\begin{align*}
D_{-i}^h\varphi(t,x):=D_{i}^{-h}\varphi(t,x),\quad 
%\\
%D^{h,2}\varphi(t,x)&:=\left(-D_{-j}^{h}D_i^h\varphi(t,x)\right)_{1\leq i,j\leq d}.  
\partial_t^\tau \varphi(t,x):=\frac{\varphi(t,x)-\varphi(t-\tau ,x)}{\tau}.
\end{align*}

\quad Now we discuss the discrete scheme. 
Given Lipschitz continuous functions $\alpha_0=\alpha_0(t,x),\beta_0=\beta_0(t,x)$, let $V_n^{\tau,h}: \Omega^{\tau,h}_{T}\to \bbR$ be defined iteratively for $n=0,1,\cdots$ as follows: 
\beq\lb{4.1}
\left\{
\begin{aligned}
&\partial_t^\tau V_n^{\tau,h} + L(t,x,\nabla^h V_n^{\tau,h} )( \alpha_n,\beta_n) =-\frac12\sum_{i=1}^d (\Sigma_i(t,x)+\nu_h )\Delta_i^h V_n^{\tau,h} &&\text{ in } \Omega^{\tau,h}_{T},\\
&V_n^{\tau,h}(T,\cdot) = g(\cdot) &&\text{ on }\bbZ^d_h,
\end{aligned}
\right.
\eeq
where $\Sigma_i$ is the $i$-th element of the diagonal matrix $\Sigma=\sigma\sigma^T$, and for $(t,x)\in \Omega^{\tau,h}_{T}$,
\begin{equation}
\lb{4.1'}
\begin{aligned}
\alpha_{n+1,b}(t,x)&\in \argmin_{a \in A} L(t,x,\nabla^h V^{\tau,h}_n(t,x))(a,b),\\
\beta_{n+1}(t,x)&\in \arg\max_{b \in B} L(t,x,\nabla^h V^{\tau,h}_n(t,x))(\alpha_{n+1,b}(t,x),b),\\
\alpha_{n+1}(t,x) &= \alpha_{n+1,\beta_{n+1}(t,x)}(t,x).    
\end{aligned}
\end{equation}
For $h\in (0,1)$, the constant $\nu_h\geq 0$ is to be selected
so that, for each $i=1,\ldots,d$,
\beq\lb{N2}
\left\{
\begin{aligned}
&\nu_h+\Sigma_{i}(t,x)\geq h\left| f_i(t,x,a,b)\right|\quad\text{ for all }(t,x,a,b),\\
&d  \nu_h+\sup_{(t,x)\in\Omega_T^{\tau,h}}\sum_{i=1}^d\Sigma_i(t,x) \leq h^2/\tau,
\end{aligned}
\right.
\eeq
which guarantees that the numerical Hamiltonian is monotone and, as a consequence of this, the comparison principle holds (see the discussion in Section \ref{sc2} and Lemma \ref{L.2.1}).

\quad We also consider the following equation
\beq\lb{4.2}
\left\{
\begin{aligned}
&\partial_t^\tau V^{\tau,h}+ H(t,x,\nabla^h V^{\tau,h}) = -\frac12\sum_{i=1}^d (\Sigma_i(t,x)+\nu_h )\Delta_i^h V^{\tau,h}\qquad &&\text{ in } \Omega^{\tau,h}_{T},\\
&V^{\tau,h}(T,\cdot) = g(\cdot)\qquad &&\text{ on }\bbZ^d_h.
\end{aligned}
\right.
\eeq
We note that both \eqref{4.1} and \eqref{4.2} are based on explicit schemes.
The goals of this paper are to show that $V_n^{\tau,h}$ converges to $V^{\tau,h}$ as $n\to\infty$ and $V^{\tau,h}$ converges to $v$ as $\tau,h\to 0$, where $v$ is given by \eqref{3.1}, and to obtain the corresponding convergence rates.

\subsection{Main results}
We first prove that $V_n^{\tau,h}$ converges to $V^{\tau,h}$ exponentially fast as $n\to \infty$. We need $\tau>0$ to be sufficiently small such that
\beq\lb{N4}
\left\{
\begin{aligned}
 &   d\tau \left(\max_{i=1,\ldots,d}\|\Sigma_i\|_\infty\right)\leq 4h^2,\\
 &  96\tau\left(\max\{\|c\|^2_\infty,2d\|f\|_\infty^2\}\right)\leq %\lambda^h:=
 \min_{i=1,\ldots,d} \inf_{(t,x)}\Sigma_i(t,x)+\nu_h.
\end{aligned}
\right.
\eeq

\begin{theorem}\lb{T.4.2}
Assume \eqref{N2} and \eqref{N4}. 
For $h\in (0,1)$, let $\Lambda^h>\lambda^h\geq \nu_h$ be such that
\[
0<\lambda^h\leq \Sigma^h_i=\Sigma_i+\nu_h\leq \Lambda^h\quad \text{ for each $i\in \{1,\ldots,d\}$}.
\]
Then, for all $n\geq 1$, 
\[
\begin{aligned}
\sup_{(t,x)\in \Omega^{\tau,h}_{T}}\left|V_n^{\tau,h}(t,x)-V^{\tau,h}(t,x)\right|^2 \leq C_h{2^{-n-1}}
\end{aligned}
\]
where $C_h:= e^{C_1T/\lambda^h}\left({12}\|g\|_\infty^2+T{\lambda^h}\right)$ and $C_1:=48\max\{\|c\|^2_\infty,2d\|f\|_\infty^2\}$.
\end{theorem}

\begin{remark}
In the case of $\nu_h=Nh$, the conditions \eqref{N2} and \eqref{N4} are satisfied when $N\geq \|f\|_\infty$, and then $h^2/\tau$ is sufficiently large.

\quad If the diffusion matrix is non-degenerate, that is, $1/c_0\geq \Lambda^h>\lambda^h\geq c_0>0$ for some $c_0\in(0,1)$ independent of $h$ and $\tau$, then the constant $C_h$ in Theorem \ref{T.4.2} can be chosen to be independent of $h$.
In this case, the exponential convergence rate of PI obtained is independent of $h$ and $\tau$.

\quad To the best of our knowledge, Theorem \ref{T.4.2} provides the first exponential convergence result for PI of possibly nonconvex viscous Hamilton--Jacobi equations in the literature.  
    
\end{remark}

\quad Since the literature is vast, we will only mention the results on PI that are directly related to \eqref{4.1}, \eqref{4.2}, and \eqref{3.1}.
 The PI method was first used to study Markov decision processes in \cite{Howard60}.
PI for deterministic optimal control problems in continuous space-time was studied in the linear quadratic setting in \cite{Klein68, VPAL09}, under specific structures allowing solvability in \cite{AL05}, and under a fixed point assumption in \cite{LS21}.
In the general setting, the problem corresponds to a first-order convex Hamilton--Jacobi equation whose solution is only Lipschitz and not smooth, and thus, the selection step in the PI similar to \eqref{eq:uiter} is not well-posed.
To overcome this ill-posedness issue, a semidiscrete scheme with an added viscosity term via finite differences in space was studied in \cite{TTZ}. 
It was proved in \cite{TTZ} that the PI for the semidiscrete scheme converges exponentially fast and the error induced by the semidiscrete scheme was provided. 
The discrete space-time scheme was also analyzed in \cite{TTZ}. 
The PI in \cite{TTZ} was incorporated with deep operator network in \cite{LeeKim24} to solve the deterministic optimal control problem and the corresponding Hamilton--Jacobi equation.

\quad For stochastic control problems with uncontrolled diffusion, \cite{KSS20} showed that PI converges exponentially fast. 
The problem corresponds to a viscous Hamilton--Jacobi equation \eqref{3.1} with a non-degenerate diffusion matrix and a convex Hamiltonian.
See  \cite{DFX24, HWZ22, TZ23, TWZ} for the corresponding entropy-regularized problems, and \cite{CCG21, C22, CT22} for PI for mean field games.
We refer to \cite{Kawecki-Smears, Kawecki-Sprekeler} for the use of PI to solve discrete problems arising in the finite element approximation of uniformly elliptic Bellman--Isaacs equations.

\quad In our setting, we face two major difficulties.
First, as the diffusion matrix can be degenerate, \eqref{eq:uiter} is not well-posed.
To handle this, we consider the discrete space-time scheme with an added viscosity term via finite differences in space, which is similar to the approach in \cite{TTZ}.
Second, our Hamiltonian $H$ is not convex in $p$ in general, which causes major challenges in obtaining the exponential convergence of the PI.
In convex case,  due to policy improvement, the value functions are monotone i.e.,  $V^{\tau,h}_{n+1} \leq V^{\tau,h}_n$ for $n\in \bbN$. The monotonicity immediately implies the convergence of PI (see \cite{TTZ}).  
Then, exponential convergence rates in $L^2$ via the energy method were obtain in  \cite{KSS20,TTZ} under the assumption that the policies are unique and Lipschitz continuous in all of their variables. 
This assumption plays a crucial role in estimating the difference in coefficients in different iterations.

\quad Of course,  the monotonicity of the value functions does not carry over to our nonconvex problem. Moreover, we assume only \eqref{eq:uiter} (or \eqref{4.1'}) without imposing any conditions on the uniqueness or regularity of the policies.
To address these issues, we carefully control the drift terms and their differences throughout the iteration process. This control relies solely on the regularity assumptions of the coefficients (Lemma \ref{L.4.1}). Additionally, we employ the maximum principle and draw inspiration from the Bernstein method (Lemma \ref{L.3.1}), leveraging the diffusion term to establish the $L^\infty$-bounds in Theorem \ref{T.4.2}.
Notably, our approach is pointwise in nature and differs from those in \cite{KSS20,TTZ}.

% \quad Theorem \ref{T.4.2} gives us the following corollary immediately.
% \begin{corollary}\lb{C.1}
% Assume the settings of Theorem \ref{T.4.2}.
% We have that
% \[
% \max_{t\in{\bbN}^\tau_{T}}\sup_{x\in\bbZ_h^d}\left|V_n^{\tau,h}(t,x)-V^{\tau,h}(t,x)\right|^2 \leq C_1h^{-d}\delta^n,
% \]
% \[
% \max_{t\in{\bbN}^\tau_{T}}\sup_{x\in\bbZ_h^d}\left|D^h(V_n^{\tau,h}(t,x)-V^{\tau,h}(t,x))\right|^2 \leq C_1\tau^{-1}h^{-d}\delta^n.
% \]
% \end{corollary}

\quad Since the policies $(\alpha_n,\beta_n)$ are not assumed to be unique, it is not possible to discuss the convergence of policies directly. Instead, we demonstrate the convergence through the Hamiltonian and the optimal value $V^{\tau,h}$.

\begin{corollary}\lb{C.2}
Assume the settings of Theorem \ref{T.4.2}.
We have, for $n\in \bbN$ and $i=1,\ldots,d$,
\[
\sup_{(t,x)\in \Omega^{\tau,h}_{T}}\left|D_i^h(V_n^{\tau,h}(t,x)-V^{\tau,h}(t,x))\right|^2 \leq  C_h h^{-2} 2^{1-n},
\]
\begin{align*}
%\sup_{(t,x)\in \Omega^{\tau,h}_{T}}\left|{L}(t,x,\nabla^h V^{\tau,h})(\alpha_{n},\beta_{n})-H(t,x,\nabla^h V^{\tau,h})\right|\leq 4\|f\|_\infty (d C_h)^{1/2} h^{-1} {2}^{(1-n)/2},
\sup_{(t,x)\in \Omega^{\tau,h}_{T}}\left|{L}(t,x,\nabla^h V^{\tau,h})(\alpha_{n},\beta_{n})-H(t,x,\nabla^h V^{\tau,h})\right|^2\leq 16dC_h\|f\|_\infty^2   h^{-2} {2}^{1-n},
\end{align*} 
where $C_h= e^{C_1T/\lambda^h}\left({12}\|g\|_\infty^2+T\lambda^h\right)$ and $C_1=48\max\{\|c\|^2_\infty,2d\|f\|_\infty^2\}$.
\end{corollary}

\quad Next, we obtain the convergence rate of \( V^{\tau,h} \) to \( v \) as $h\to 0$.

\begin{theorem}\lb{T.4.3}
Assume \eqref{N2}. 
Assume further that $c(t,x,\alpha,\beta),\,f(t,x,\alpha,\beta),\,\sigma(t,x)$ are uniformly Lipschitz in $(t,x)$, and $\|g\|_{C^3}<\infty$.
Let \( V^{\tau,h} \) and \( v \) be the solutions to \eqref{4.2} and \eqref{3.1}, respectively. 
Fix any $\alpha\in(0,1)$.
\begin{itemize}
    \item[(a)] If the diffusion matrix $\Sigma$ is non-degenerate, that is, $1/c_0\geq \Sigma_i\geq c_0$ for some $c_0\in(0,1)$ independent of $h$ and $\tau$ and for all $i=1,\ldots, d$, then 
\[
\sup_{(t,x)\in \Omega^{\tau,h}_T} |V^{\tau,h}(t,x)-v(t,x)|\leq
Ch^{\alpha/2}.
\]
Here, $C>0$ depends only on $d$, $\alpha$, $c_0$, $\|g\|_{C^3}$,  $\|c\|_{\rm Lip}$, $\|f\|_{\rm Lip}$,  and $\|\sigma\|_{\rm Lip}$.

\item[(b)] If $\Sigma$ is degenerate, that is, $\min_{1\leq i\leq d} \min_{(t,x)} \Sigma_i(t,x)=0$, then
\[
\sup_{(t,x)\in \Omega^{\tau,h}_T} |V^{\tau,h}(t,x)-v(t,x)|\leq
Ch^{2\alpha/({9+7\alpha})}.
\]
Here, $C>0$ depends only on $d$, $\alpha$, $\|g\|_{C^3}$,  $\|c\|_{\rm Lip}$, $\|f\|_{\rm Lip}$,  and $\|\sigma\|_{\rm Lip}$.
\end{itemize}

\end{theorem}

\begin{remark}

If the diffusion matrix $\Sigma$ is non-degenerate, then the convergence rate is  $O(h^{\alpha/2})$ for each $\alpha \in (0,1)$, which is close to the optimal convergence rate $O(h^{1/2})$ for convex equations in \cite{10DonKry,DonKry,18Kry}.

\quad And in the case when $\Sigma$ is degenerate, the convergence rate is $O\left(h^{2\alpha/({9+7\alpha})}\right)$.
This convergence rate is close to $O(h^{1/8})$ as $\alpha \to 1$.

\quad  To the best of our knowledge, the convergence results in Theorem \ref{T.4.3} are new in the literature.
\end{remark}

\quad Quantitative convergence results of discrete space-time schemes for Hamilton--Jacobi equations is a popular topic, and we will only mention the results that are directly related to \eqref{4.2} and \eqref{3.1}.
For first-order equations, the optimal convergence rate $O(h^{1/2})$ was obtained in \cite{crandall1984two}.
For second-order equations, the problem becomes much more complicated because of the appearance of the diffusion term.
In a way, the main challenge is the lack of appropriate regularizations of viscosity solutions yielding control on derivatives higher than two.
For convex equations which can be degenerate, we refer the reader to \cite{Barles-Jakobsen,10DonKry,DonKry,18Kry}, in which the optimal convergence rate $O(h^{1/2})$ was proved.
For fully nonlinear uniformly elliptic/parabolic equations which can be nonconvex, algebraic convergence rates $O(h^\theta)$ were obtained in \cite{Caffarelli-Souganidis,turanova2015error, Turanova}.
For general fully nonlinear equations which can be both degenerate and nonconvex, the problem remains largely open (see \cite{Jakobsen} for a special case in one dimension). 
We also mention that for a nonlocal Isaacs equation, \cite{biswas2019rate} established the convergence with rate depending on the nonlocal kernel, using the method of doubling variables in \cite{crandall1984two}.

\quad In our setting, \eqref{3.1} is possibly both degenerate and nonconvex, but is linear in the second-order term, which is hence simpler than the most general fully nonlinear equations.
To overcome the lack of the bounds of derivatives higher than two of the viscosity solutions, we consider the approximate equation \eqref{9.1} to \eqref{3.1} and obtain bounds on the H\"older norm of the gradient of the solution $v^{\delta,\eps}$.
We then perform a convolution regularization of $v^{\delta,\eps}$ to get $u^\delta_\eps$ and have the bounds of derivatives higher than two of $u^\delta_\eps$.
This allows us to compare $u^\delta_\eps$ to $V^{\tau,h}$, the solution of \eqref{4.2}, which is the main step to prove Theorem \ref{T.4.3}.

\subsection*{Organization of the paper}
In Section \ref{sc2}, we provide some preliminaries on the monotonicity of the discrete space-time schemes, the comparison principle, and some estimates.
The proof of Theorem \ref{T.4.2} is given in Section \ref{S3}.
The convergence of the discrete equations and the proof of Theorem \ref{T.4.3} are given in Section \ref{S4}.

%-------------------------------------------------------------------------------------------------
\section{Preliminaries}
\label{sc2}

\quad We are concerned with the monotonicity of the discrete space-time schemes.
%\subsection{Monotonicity of the schemes}
We will use the following operator. For each $t\in 
{\bbN}^\tau_{T}$ and $(a,b)\in A\times B$, let $\calF_t^{a,b}: L^\infty(\bbZ_h^d)\to L^\infty(\bbZ_h^d)$ be defined as
\[%\beq\lb{Fab}
\calF_t^{a,b}(U)(x):=U(x)+\tau L(t,x,\nabla^h U(x))(a,b)+\frac\tau2\sum_{i=1}^d(\Sigma_i+\nu_h) \Delta^{h}_i U(x).
\]%\eeq
Then the equation in \eqref{4.1} can be rewritten as
\[
V^{\tau,h}_n(t-\tau,x)=\calF_t^{\alpha_n,\beta_n}(V_n^{\tau,h}(t,\cdot))(x)
\]
or, more precisely,
\begin{align*}
V^{\tau,h}_n(t-\tau,x)%&=V^{\tau,h}_n(t,x)+\tau c^n(t,x) + \tau\sum_{i=1}^d f_i^n(t,x)\frac{V^{\tau,h}_n(x+he_i)-V^{\tau,h}_n(x-he_i)}{2h}\\
%&\quad+\frac\tau2\Tr \left[\Sigma^n(t,x) \left(\frac{V^{\tau,h}_n(x+he_i)-2V^{\tau,h}_n(x)+V^{\tau,h}_n(x-he_j))}{h^2}\right)_{1\leq i,j\leq d}\right]\\
%&\quad+ N\tau\sum_{i=1}^d\frac{V^{\tau,h}_n(t,x+he_i)-2V^{\tau,h}_n(t,x)+V^{\tau,h}_n(t,x-he_i)}{h}\\
&=\tau c^n(t,x)+V^{\tau,h}_n(t,x)\left[1-\tau h^{-2}\sum_{i=1}^d\left(\Sigma_i(t,x)+\nu_h\right)\right]\\ 
&\quad+ \frac{\tau}{2}\sum_{i=1}^d {V^{\tau,h}_n(x+he_i)}\Big[ hf_i^n(t,x)+\Sigma_{i}(t,x)+\nu_h\Big]\\%+\frac{\tau }{2h^2}\Tr \left[\Sigma^n(t,x) \left({V^{\tau,h}_n(x+he_i)}\right)_{1\leq i,j\leq d}\right]\\
&\quad + \frac{\tau}{2}\sum_{i=1}^d {V^{\tau,h}_n(x-he_i)}\Big[-hf_i^n(t,x)+\Sigma_{i}(t,x)+\nu_h\Big]
\end{align*}
where
\[
f_i^n(t,x):=f_i(t,x,\alpha_n,\beta_n).
\]

The above formula shows that if \eqref{N2} holds, that is, for each $i=1,\ldots,d$,
\begin{equation*}
\left\{
\begin{aligned}
&\nu_h+\Sigma_{i}(t,x)\geq h\left| f_i(t,x,a,b)\right|\quad\text{ for all }(t,x,a,b),\\
&d  \nu_h+\sup_{(t,x)\in\Omega_T^{\tau,h}}\sum_{i=1}^d\Sigma_i(t,x) \leq h^2/\tau,
\end{aligned}
\right.
\end{equation*}
then the following monotonicity formula holds. 
For all $t\in{\bbN}^\tau_{T}$, $a,b\in A\times B$ and $U,V\in L^\infty(\bbZ_h^d)$  satisfying $U\leq V$,
\[
\calF^{a,b}_t(U)\leq \calF^{a,b}_t(V).
\]
We also refer readers to  \cite{Barles-Souganidis, crandall1984two, osher1991, tranbook}.

\quad Similarly, if we define
\beq\lb{calF}
\calF_t(U)(x):=U(x)+\tau H(t,x,\nabla^h U(x))+\frac\tau2\sum_{i=1}^d(\Sigma_i+\nu_h) \Delta^{h}_i U(x),
\eeq
then $V^{\tau, h}$, solving \eqref{4.2}, satisfies
\[
V^{\tau,h}(t-\tau,x)=\calF_t(V^{\tau,h}(t,\cdot))(x)
\]
and
\begin{equation*}%\label{eq:monotone}
\calF_t(U)\leq \calF_t(V)\quad\text{ whenever }\quad U\leq V.
\end{equation*}

\begin{definition}
We say that $V$ is a supersolution (resp. subsolution) to \eqref{4.1} or \eqref{4.2} if it satisfies \eqref{4.1} or \eqref{4.2} with the first equality replaced by $\leq$ (resp. $\geq$) and the second equality replaced by $\geq$ (resp. $\leq$).   
\end{definition}

The following is the comparison principle in this discrete space-time setting.

\begin{lemma}\lb{L.2.1}
Assume \eqref{N2}.
Let $V^{\tau,h}$ and $\tilde{V}^{\tau,h}$ be, respectively, a bounded supersolution and subsolution to \eqref{4.1} with $n=0$ or \eqref{4.2}. % given terminal data $q$ and $\tilde{q}$ satisfying $\tilde{q}\leq q$ in $\bbR^d$. 
Then $\tilde{V}^{\tau,h}\leq {V}^{\tau,h}$ in $\Omega_T^{\tau,h}$.  
\end{lemma}
\begin{proof}
Let us only consider the case for equation \eqref{4.2}.  By the assumption,
\[
\tilde V^{\tau,h}(T,\cdot)\leq V^{\tau,h}(T,\cdot).
\]
Let $\calF_t$ from \eqref{calF} and then by the definitions of supersolution and subsolution,
\[
V^{\tau,h}(t-\tau,x)\geq \calF_t(V^{\tau,h}(t,\cdot))(x)\quad\text{and}\quad \tilde V^{\tau,h}(t-\tau,x)\leq \calF_t(\tilde V^{\tau,h}(t,\cdot))(x).
\]
These inequlities and the monotonicity property of $\calF_t$ yield
\[
\tilde V^{\tau,h}(T-\tau,x)\leq \calF_T(\tilde V^{\tau,h}(T,\cdot))(x)\leq \calF_T(V^{\tau,h}(T,\cdot))(x)\leq V^{\tau,h}(T-\tau,x).
\]
Thus, by induction, we obtain for all $k\in \{1,\ldots,T/\tau-1\}$ that
\begin{align*}
\tilde V^{\tau,h}(T-(k+1)\tau,x)&\leq \calF_{T-k\tau}(\tilde V^{\tau,h}({T-k\tau},\cdot))(x)\\
&\leq \calF_{T-k\tau}(V^{\tau,h}({T-k\tau},\cdot))(x)\leq V^{\tau,h}(T-(k+1)\tau,x).
\end{align*}
\end{proof}

\quad Of course, the monotonicity property and the corresponding maximum principle play a crucial role in our analysis.

%\subsection{Preliminaries}

\medskip

\quad The following lemma proves that $V^{\tau,h}, V^{\tau,h}_n$ are uniformly bounded. %, and that $V^{\tau, h}$ is uniformly Lipschitz continuous in the domain $\Omega^{\tau,h}_T$. 
%-------------------------------------------------------------------------------------------------

\begin{lemma}%\lb{L.2.0}
Assume \eqref{N2}. Let $V_n^{\tau,h}$ solve \eqref{4.1}--\eqref{4.1'} and $V^{\tau,h}$ solve \eqref{4.2}. Then in the domain of $\Omega^{\tau,h}_T$, $V^{\tau,h}$ and $V_n^{\tau,h}$ for all $n\geq 0$ are uniformly bounded by $\|g\|_\infty+\|c\|_\infty T$.
\end{lemma}

\begin{proof}
First we prove the boundedness of $V_n^{\tau,h}$. %Let $F_t^{\alpha,\beta}$ from \eqref{Fab} and then
%\[
%V^{\tau,h}_n(t-\tau,x)=\calF_t^{\alpha_n,\beta_n}(V^{\tau,h}_n(t,\cdot))(x).
%\]
Since $c$ and $g$ are uniformly bounded, 
\[\pm \left[\|g\|_\infty+\|c\|_\infty(T-t)\right]
\]
are a supersolution and a subsolution to \eqref{4.1} for any $n\geq 1$, respectively. Thus the comparison principle (Lemma \ref{L.2.1}) implies that
\[
-\|g\|_\infty-\|c\|_\infty(T-t)\leq V^{\tau,h}_n(t,x)\leq \|g\|_\infty+\|c\|_\infty(T-t).
\]
By the same argument, the same estimate holds if we replace $V^{\tau,h}_n$ by $V^{\tau,h}$.
\end{proof}

\quad The following lemma concerns the regularity property of $H$.

\begin{lemma}\lb{L.2.3}
For any $(t_i,x_i)\in\Omega_T^{\tau,h}$ and $p_i\in \bbR^d$ with $i=1,2$, we have
\begin{align*}
&\left|H(t_1,x_1,p_1)-H(t_2,x_2,p_2)\right|\\
&\qquad\leq \left(\|c\|_{\Lip}+\|f\|_{\Lip}\min\{|p_1|,|p_2|\}\right)(|t_2-t_1|+|x_2-x_1|)+\|f\|_\infty|p_1-p_2|.
\end{align*}    
\end{lemma}
\begin{proof}
Let $(\alpha_i,\beta_i)$ for $i=1,2$ be such that
\[
H(t_i,x_i,p_i)={L}(t_i,x_i,p_i)(\alpha_{i},\beta_{i}).
\]
For any $b\in B$,  define $\alpha_{i,b}$ as
\[
\alpha_{i,b}(t,x)\in \argmin_{a \in A} L(t_i,x_i,p_i)(a,b).
\]
By \eqref{3.1'} and the regularity assumptions on $c$ and $f$, it follows that
\[
\begin{aligned}
& {L}(t_1,x_1,p_1)(\alpha_{1},\beta_{1})-{L}(t_2,x_2,p_2)(\alpha_2,\beta_2)\\
\geq\, & {L}(t_1,x_1,p_1)(\alpha_{1,\beta_2},\beta_2)-{L}(t_2,x_2,p_2)(\alpha_{1,\beta_2},\beta_2)\\
 =\, &c(t_1,x_1,\alpha_{1,\beta_2},\beta_2) + p_1 \cdot f(t_1,x_1,\alpha_{1,\beta_2},\beta_2)-c(t_2,x_2,\alpha_{1,\beta_2},\beta_2) - p_2 \cdot f(t_2,x_2,\alpha_{1,\beta_2},\beta_2)\\
\geq\,& -(\|c\|_{\Lip}+\|f\|_{\Lip}\min\{|p_1|,|p_2|\})(|t_2-t_1|+|x_2-x_1|)-\|f\|_\infty|p_1-p_2|.
\end{aligned}  
\]
The other direction follows in the same manner.
\end{proof}

\quad The following lemma estimates the difference between 
\[
{L}(t,x,\nabla^h V_n^{\tau,h})(\alpha_{n},\beta_{n})\quad\text{and}\quad H(t,x,\nabla^h V^{\tau,h})
\]
which, respectively, appeared in \eqref{4.1} and \eqref{4.2}.

\begin{lemma}\lb{L.4.1}
For $n\geq 1$,
\begin{align*}
\left|{L}(t,x,\nabla^h V_n^{\tau,h})(\alpha_{n},\beta_{n})-H(t,x,\nabla^h V^{\tau,h})\right|\leq \|f\|_\infty\left[|\nabla^h (V_n^{\tau,h}-V^{\tau,h})|+|\nabla^h (V_{n-1}^{\tau,h}-V^{\tau,h})|\right].
\end{align*}    
\end{lemma}
\begin{proof}
For fixed $(t,x)$, let $(\alpha_*,\beta_*)$ be such that
\[
{L}(t,x,\nabla^h V^{\tau,h})={L}(t,x,\nabla^h V^{\tau,h})(\alpha_*,\beta_*).
\]
It follows from Lemma \ref{L.2.3} that
\beq\lb{3.3}
\begin{aligned}
\left| {L}(t,x,\nabla^h V_{n-1}^{\tau,h})(\alpha_{n},\beta_{n})-{L}(t,x,\nabla^h V^{\tau,h})(\alpha_*,\beta_{*})\right|\leq  \|f\|_\infty |\nabla^h (V_{n-1}^{\tau,h}-V^{\tau,h})|.
\end{aligned}  
\eeq
Since $L(t,x,p)(a,b)$ is Lipschitz continuous in $p$, by using \eqref{3.3} and the triangle inequality, we obtain
\begin{align*}
& \left|{L}(t,x,\nabla^h V_n^{\tau,h})(\alpha_{n},\beta_{n})-{L}(t,x,\nabla^h V^{\tau,h})(\alpha_*,\beta_*)\right|\\
\leq\, &\left| {L}(t,x,\nabla^h V_{n-1}^{\tau,h})(\alpha_{n},\beta_{n})-{L}(t,x,\nabla^h V^{\tau,h})(\alpha_*,\beta_*)\right|+\|f\|_\infty |\nabla^h (V_n^{\tau,h}-V_{n-1}^{\tau,h})|\\
%&\qquad\quad-C|\nabla^h (V_{n-1}^{\tau,h}-V^{\tau,h})|\\
\leq\, & \|f\|_\infty|\nabla^h (V_n^{\tau,h}-V^{\tau,h})|+\|f\|_\infty|\nabla^h (V_{n-1}^{\tau,h}-V^{\tau,h})|.
\end{align*}    
\end{proof}

\section{Proof of Theorem \ref{T.4.2}}
\label{S3}

%We assume for now that $V_0^{\tau,h}$ and $V^{\tau, h}$ are uniformly Lipschitz continuous in the domain $\Omega^{\tau,h}_T$.

\quad Throughout this section, we always assume the settings of Theorem \ref{T.4.2}.
We write $V_n:=V^{\tau,h}_n$ and $V_*:=V^{\tau,h}$, which are, respectively, bounded solutions to \eqref{4.1} and \eqref{4.2}. For simplicity of notation, we write 
\[
\alpha_n:=\alpha_n(t,x),\qquad
\alpha_*:=\alpha(t,x,\nabla^h V_*(t,x)),
\]
and similarly for $\beta_n$ and $\beta_*$. %Denote
%\[
%{L}(t,x,p)(a,b): =c(t,x,a,b) + p \cdot f(t,x,a,b),
%\]
%\[
%{H}(t,x,p):= \sup_{b\in B}\inf_{a \in A} {L}(t,x,p)(a,b)
%\]
Recall the notations of \eqref{3.1'} and write 
\[
\Sigma^h_i:=\Sigma_i+\nu_h. %\qquad{\text{\bf Maybe just $\Sigma=\tilde \Sigma$...}}
\]
%Then \eqref{4.2} can be rewritten as
%\beq\lb{4.2}
%\left\{
%\begin{aligned}
%&\partial_t^\tau V_*+ {H}(t,x,\nabla^h V_*)(\alpha_*,\beta_*) =-\frac12\sum_{i=1}^d\Sigma^h_i\Delta^{h}_i V_* \qquad &&\text{ in } \Omega^{\tau,h}_{T},\\
%&V_*(T,\cdot) = g(\cdot)\qquad &&\text{ on }\bbZ^d_h
%\end{aligned}
%\right.
%\eeq
%and \eqref{4.1} becomes
%\beq\lb{4.1}
%\left\{
%\begin{aligned}
%&\partial_t^\tau V_n +{L}(t,x,\nabla^h V_n )( \alpha_n,\beta_n) = - \sum_{i=1}^d\Sigma^h_i \Delta^{h}_i  V_n &&\text{ in } \Omega^{\tau,h}_{T},\\
%&V_n(T,\cdot) = g(\cdot) &&\text{ on }\bbZ^d_h.
%\end{aligned}
%\right.
%\eeq
We also denote
\[
 c_n:=c(t,x,\alpha_n,\beta_n) \quad \text{ and } \quad f_n:=f(t,x,\alpha_n, \beta_n)
\]
and we will sometimes drop $(t,x)$ from the notations of $V_n(t,x)$ and $V_*(t,x)$. We will write
\[
\calE:=\{ \pm 1,\,\ldots\,,\pm d\}.
\]

\quad The main goal of this section is to prove an exponential convergence rate of PI for the discrete space-time schemes, Theorem \ref{T.4.2}.
We emphasize that the drift terms do not help establish the convergence of $V_n$ to $V_*$.
Instead, the key terms for obtaining a quantitative bound of $V_n-V_*$ are the diffusion terms. 
%Therefore, we employ the energy method to derive the $L^2$-bounds for $V_n-V_*$ and $D^h(V_n-V_*)$.

%\quad The definition of Since $\phi_R(x-he_i)\leq (1+Ch)\phi_R(x)$ by the choice of $\phi_R$, 

\begin{lemma}\label{L.3.1}
Assume for each $i=1,\ldots, d$, $M_i=M_i(t,x)\in [\lambda,\Lambda]$ for some $0<\lambda\leq\Lambda<\infty$, and
\begin{equation}\lb{N3}
d\Lambda\tau\leq 4h^2.
\end{equation}
Let $u,\ell$ be functions on $\Omega^{\tau,h}_{T}$ that satisfy 
\[
\calL(u):=\timed u+\frac{1}{2}\sum_{i=1}^d M_i{\Delta^{h}_i} u=\ell \quad \text{ in } \Omega^{\tau,h}_{T}.
\]
Then, we have
\[
\calL(u^2)\geq 2u\ell-2\tau \ell^2+\frac\lambda{4}\sum_{j\in\calE}|D^h_j u|^2 \quad \text{ in } \Omega^{\tau,h}_{T}.
\]
\end{lemma}
\begin{proof}
Note that
\begin{align*}
\timed(u^2)&=(\timed u)u+u(t-\tau,x)\timed u=2u\,\timed u-\tau (\timed u)^2,
\\
D_i^h(u^2)&= (D_i^hu)u+u(t,x+he_i)D_i^h u=2uD_i^h u+h(D_i^hu)^2.    
\end{align*}
Hence, by \eqref{p.9},
\begin{align*}
M_i \Delta_i^h (u^2)=  2u M_i\Delta_i^h u+M_i h \sum_{j=\pm i}(D_j^h u)^2.
\end{align*}
Writing $M_{-i}:=M_i$, we get
\beq\lb{*3}
\calL(u^2)=2u\calL u-\tau (\timed u)^2+\frac12\sum_{j\in\calE}M_j(D_j^h u)^2.
\eeq
Furthermore, note that
\begin{align*}
\left(\partial_t^\tau u\right)^2&=\left(\ell-\frac12\sum_{i=1}^d M_i\Delta_i^h u\right)^2\leq 2\ell^2+\frac{1}{2}\left(\frac{1}{h}\sum_{j\in\calE}^d M_jD_j^h u\right)^2\\
&\leq 2\ell^2+\frac{1}{2h^2}\left(\sum_{j\in\calE}M_j(D_j^h u)^2\right)\left(\sum_{j\in\calE}M_j\right)\leq 2\ell^2+\frac{d\Lambda}{h^2}\sum_{j\in\calE}M_j(D_j^h u)^2.
\end{align*}
Thus, it follows from \eqref{*3} and \eqref{N3} that
\[
\calL(u^2)\geq 2u\ell-2\tau\ell^2+\left(\frac12-\frac{d\Lambda\tau}{h^2}\right)\sum_{j\in\calE} M_j(D_j^h u)^2\geq 2u\ell-2\tau\ell^2+\frac14\sum_{j\in\calE} M_j(D_j^h u)^2,
\]
which finishes the proof.
\end{proof}

\quad Now we prove Theorem \ref{T.4.2}.

\begin{proof}[{Proof of Theorem \ref{T.4.2}}]
Let us write $\lambda:=\lambda_h$ and $\Lambda:=\Lambda_h$. Then the condition \eqref{N4} implies \eqref{N3} and $96\tau \left(\max\{\|c\|_\infty^2,2d\|f\|_\infty^2\}\right)\leq \lambda$.
Setting $u_n:=V_n-V_*$, we get
   \begin{align*}
        \begin{cases}
       \timed u_n+\frac12\sum_{i=1}^d \Sigma^h_i\Delta_i^h u_n=\ell_n \quad & \text{ in }\Omega^{\tau,h}_{T}, \\
       u_n(T,\cdot)=0 \quad & \text{ on }\bbZ_h^d
    \end{cases}
   \end{align*}
where
\[
\ell_n=H(t,x,\nabla^h V_*)-L(t,x,\nabla^h V_n^{\tau,h})(\alpha_n,\beta_n).
\]
By Lemma \ref{L.4.1},
\beq\lb{4.41}
|\ell_n|^2\leq \|f\|_\infty^2\left(|\nabla^h u_n|+|\nabla^h u_{n-1}|\right)^2\leq C_0\sum_{j\in\calE}\left(|{D_j^h} u_n|^2+|{D_j^h} u_{n-1}|^2\right)
\eeq
with $C_0:=4d\|f\|_\infty^2$. Define
\[
G_n:=G_n(t,x)=\sum_{j\in\calE}|D^h_j u_n|^2(t,x).
\]
By Lemma \ref{L.3.1}, for any $\eps_0>0$ we have
\begin{align*}
    \calL(u_n^2)&\geq 2u_n\ell_n-2\tau \ell_n^2+\frac\lambda{4}G_n\geq -\frac1{\eps_0}u_n^2-(\eps_0+2\tau) \ell_n^2+\frac\lambda{4}G_n\\
    &\geq -\frac1{\eps_0}u_n^2 +\left(\frac\lambda{4}-C_0\eps_0-2C_0\tau\right)G_n-C_0(\eps_0+2\tau)G_{n-1}
\end{align*}
where in the last inequality, we used \eqref{4.41}.
Choosing $\eps_0:=\frac{\lambda}{24C_0}$, and noting $\tau\leq \frac
\lambda{48C_0}$, we obtain
\beq\lb{*2}
\calL(u_n^2)\geq -\frac1{\eps_0}u_n^2 +\frac\lambda{6}G_n-\frac{\lambda}{12}G_{n-1}.
\eeq

\quad Next, define
\beq\lb{xi}
\xi(t):=(1+\tau/\eps_0)^{t/\tau}.
\eeq
Direct computation yields $\timed(\xi)(t)=\xi(t-\tau)/\eps_0$, and
\[%beq\label{eq:chainrule}
\begin{aligned}
    \timed(\xi\,u_n^2)
    &=\xi(t-\tau)\timed(u_n^2)+\timed(\xi) u_n^2(t,x)\\
    &=\xi(t-\tau)\timed(u_n^2)+\xi(t-\tau)\frac{1}{\eps_0}u_n^2(t,x).
\end{aligned}
\]
Hence, \eqref{*2} implies
\beq\lb{*6}
\calL(\xi\, u_n^2)\geq  \frac{\lambda\xi(t-\tau)}{12}\left(2G_n-G_{n-1}\right).
\eeq

\quad Now, let $\psi_n$ be the solution of
\[%\beq\lb{*6.5}
        \begin{cases}
       \calL \psi_n=-\xi(t-\tau)G_n \quad & \text{ in }\Omega^{\tau,h}_{T}, \\
       \psi_n(T,\cdot)=0 \quad & \text{ on }\bbZ_h^d.
    \end{cases}
\]%\eeq
Since $G_n\geq 0$, by the maximum principle  (see Lemma \ref{L.2.1} with $L=0$), $\psi_n\geq 0$.
Then \eqref{*6} shows
\[
        \begin{cases}
       \calL \left(\xi(t)u_n^2+\frac{\lambda}{6}\psi_n-\frac{\lambda}{12}\psi_{n-1}\right)\geq 0\quad & \text{ in }\Omega^{\tau,h}_{T}, \\
       \xi(T)u_n^2(T,\cdot)+\frac{\lambda}{6}\psi_n(T,\cdot)-\frac{\lambda}{12}\psi_{n-1}(T,\cdot)=0 \quad & \text{ on }\bbZ_h^d.
    \end{cases}
\]
By the maximum principle again, we get
\[%beq\lb{*7}
\xi(t)u_n^2(t,\cdot)+\frac{\lambda}{6}\psi_n(t,\cdot)-\frac{\lambda}{12}\psi_{n-1}(t,\cdot)\leq 0 \quad \text{ in }\Omega^{\tau,h}_{T}.
\]
Consequently, we find for all $n\geq 1$,
\[
\psi_n\leq \frac12\psi_{n-1}\leq\ldots\leq 2^{-n}\psi_0.
\]
Moreover, \eqref{*6} yields
\[
\xi(t)u_n^2(t,x)\leq \frac{\lambda}{12}\psi_{n-1}(t,x)\leq  \frac{\lambda}{6}2^{-n}\psi_{0}(t,x).
\]
Thus
\begin{align*}
    |V_n-V_*|^2\leq \frac{\lambda}{6}2^{-n}\psi_{0}(t,x).
\end{align*}

\quad Finally, since 
\[
\psi_0\leq \psi_{0,1}+\psi_{0,2}
\]
where for $i=1,2$, $\psi_{0,i}$ is the solution to
\beq\lb{*8}
\begin{cases}
\calL \psi_{0,i}=-2\xi(t-\tau)\sum_{j\in\calE}|D_j^h v_i|^2 \quad & \text{ in }\Omega^{\tau,h}_{T}, \\
\psi_{0,i}(T,\cdot)=0 \quad & \text{ on }\bbZ_h^d.
\end{cases}
\eeq
Here, $v_1=V_*$ and $v_2=V_0$.
The proof is finished after invoking Lemma \ref{L.3.2} below.
\end{proof}

\begin{lemma}\label{L.3.2}
For $i=1,2$, let $\psi_{0,i}$ solve \eqref{*8}. Then
\[
\psi_{0,i}\leq e^{C_1T/\lambda}\left(\frac{12}\lambda \|g\|_\infty^2+T\right)\qquad\text{ in }\Omega^{\tau,h}_{T}
\]
where $C_1:=48\max\{\|c\|_\infty^2,\,2d\|f\|_\infty^2\}$.
\end{lemma}
\begin{proof}
Let us only prove the result for $i=1$ and $u=\psi_{0,1}$.
By the equation \eqref{4.2}, $V_*:=V^{\tau,h}$ satisfies
\[
\partial_t^\tau V_* +\frac12 \sum_{i=1}^d {\Sigma^h_i}\Delta V_*=\ell
\]
where $\ell:=H(t,x,\nabla^h V_*)$. We have
\[
|\ell|^2\leq 2\|c\|_\infty^2+2\|f\|_\infty^2|\nabla^h V_*|^2\leq 2\|c\|_\infty^2+4d\|f\|_\infty^2 \sum_{j\in\calE}|D_j^h V_*|^2.
\]
Setting $C_0:=\max\{2\|c\|_\infty^2,\,4d\|f\|_\infty^2\}$ and $G:=\sum_{j\in\calE}|D_j^h V_*|^2$, Lemma \ref{L.3.1} yields that
for any $\eps_0>0$,
\begin{align*}
    \calL(V_*^2)&\geq 2V_*\ell-2\tau \ell^2+\frac\lambda{4}G\geq -\frac1{\eps_0}V_*^2 \ell-(\eps_0+2\tau) \ell^2+\frac\lambda{4}G\\
    &\geq -\frac1{\eps_0}V_*^2 +\left(\frac\lambda{4}-C_0\eps_0-2C_0\tau\right)G-C_0(\eps_0+2\tau).
\end{align*}
Taking $\eps_0:=\frac{\lambda}{24C_0}$ and using $\tau\leq \frac
\lambda{48C_0}$, we obtain
\[
 \calL(V_*^2)\geq -\frac1{\eps_0}V_*^2 +\frac\lambda{6}G-\frac{\lambda}{12}, 
\]
which yields 
\[
 \calL(V_*^2+t\lambda/12)\geq -\frac1{\eps_0}V_*^2 +\frac\lambda{6}G. 
\]
Recall $\xi$ from \eqref{xi}. This yields
\[
\calL\left[\xi\,(V_*^2+t\lambda/12)\right]\geq \lambda \xi(t-\tau)G/6.
\]

\quad Since $\psi_{0,1}$ satisfies \eqref{*8} with $v_1:=V_*$, then for 
\[
v(t,x):=\xi(t)\,(12V_*^2/\lambda+t),
\]
we find
\[
\calL\left(\psi_{0,1}+v\right)\geq 0.
\]
The maximum principle yields
\[
(\psi_{0,1}+v)(t,\cdot)\leq (\psi_{0,1}+v)(T,\cdot)= \xi(T)(12 g^2/\lambda+T).
\]
Since $\xi(T)\leq e^{T/\eps_0}$, we proved the conclusion for $\psi_{0,1}$.
\end{proof}

\begin{proof}[{Proof of Corollary \ref{C.2}}]
It follows from Lemmas \ref{L.4.1} and \ref{L.2.3} that
\begin{align*}
&\left|{L}(t,x,\nabla^h V_*)(\alpha_{n},\beta_{n})-H(t,x,\nabla^h V_*)\right|\\
\leq\,& \left|{L}(t,x,\nabla^h V_*)(\alpha_{n},\beta_{n})-{L}(t,x,\nabla^h V_n)(\alpha_{n},\beta_{n})\right|\\
&\qquad \qquad \qquad\qquad\qquad +\left|{L}(t,x,\nabla^h V_n)(\alpha_{n},\beta_{n})-{L}(t,x,\nabla^h V_*)(\alpha_*,\beta_*)\right|\\
\leq\,& 2\|f\|_\infty\left[|\nabla^h (V_n-V_*)|+|\nabla^h (V_{n-1}-V_*)|\right].
\end{align*} 
By Theorem \ref{T.4.2}, we have for each $i=1,\ldots,d$,
\begin{align*}
\sup_{(t,x)\in \Omega^{\tau,h}_{T}}\left|D_i^h(V_n(t,x)-V_*(t,x))\right|^2 &\leq \frac{4}{h^2}\sup_{(t,x)\in \Omega^{\tau,h}_{T}}\left|V_n(t,x)-V_*(t,x)\right|^2\\
&\leq  \frac{\lambda  e^{C_1T/\lambda }}{2^{n-1} h^2}\left(\frac{12}{\lambda }\|g\|_\infty^2+T\right). 
\end{align*}
Combining the two estimates yields the conclusion.
\end{proof}

\begin{comment}
We use Proposition \ref{prop:bound-G^h-s-0} to get
\[
\begin{aligned}
\frac{h^d}{R^d}\sum_{x\in\bbZ_h^d, |x|\leq R}\left|V_n^{\tau,h}(t,x)-V^{\tau,h}(t,x)\right|^2 \leq C_1 \delta^n
\end{aligned}
\]
and \eqref{3.9} to have
\[
\frac{\tau h^d}{R^d}\sum_{ s\in {\bbN}^\tau_{T}}\sum_{x\in\bbZ_h^d, |x|\leq R}\left|D^h(V_n^{\tau,h}(s,x)-V^{\tau,h}(s,x))\right|^2 \leq C_1\delta^n.
\]
The claim holds after shifting the $x$ variable and taking $R=1$.
\end{comment}

\quad Let us consider an example to see how the PI works.
\begin{ex}
    Let us consider
    \[
    H(t,x,p,X)=\max\{|p|-1,1-|p|\}+ V(x) + \frac12\Tr ((\sigma\sigma^T)(t,x) X).
    \]
    We can write
    \[
    \max\{|p|-1,1-|p|\}+ V(x) = \max_{\substack{i=1,2\\|e|\leq 1}}\min_{|a|\leq 1} \left(c(x,i)+f(a,(i,e))\cdot p  \right).
    \]
    Here,
    \[
    f(a,(1,e))=a, \qquad f(a,(2,e))=e,
    \]
    and
    \[
    c(x,1)=1+V(x),\qquad c(x,2)=-1+V(x).
    \]
    Denote by $b=(i,e)\in \{1,2\}\times \overline B_1$.
    Set $A=\overline B_1$ and $B=\{1,2\}\times \overline B_1$.

\quad     By computations, we choose that
    \[
    \alpha_b(t,x,p,X)=\alpha_b(p) = \alpha(p)=
    \begin{cases}
        -\frac{p}{|p|} \qquad &\text{ for } p \neq 0,\\
        0 \qquad &\text{ for } p =0,
    \end{cases}
    \]
    and
        \[
    \beta(t,x,p,X)=\beta(p) = 
    \begin{cases}
        (1,0) \qquad &\text{ for } |p| \leq 1,\\
        \left(2,\frac{p}{|p|}\right) \qquad &\text{ for } |p|\geq 1.
    \end{cases}
    \]
    In the above, we can select $\alpha_b=\alpha$, that is, $\alpha$ is independent of $b$.
    It is clear that $\alpha$ is discontinuous at $p=0$, and $\beta$ is discontinuous at $|p|=1$.
    Then, in the iterative process,
    \[
    \alpha_{n+1}(t,x)=
    \begin{cases}
        -\frac{\nabla^h V_n}{|\nabla^h V_n|} \qquad &\text{ for } \nabla^h V_n \neq 0,\\
        0 \qquad &\text{ for } \nabla^h V_n =0,
    \end{cases}
    \]
    and
     \[
    \beta_{n+1}(t,x) = 
    \begin{cases}
        (1,0) \qquad &\text{ for } |\nabla^h V_n| \leq 1,\\
        \left(2,\frac{\nabla^h V_n}{|\nabla^h V_n|}\right) \qquad &\text{ for } |\nabla^h V_n|\geq 1.
    \end{cases}
    \]
    In particular,
    \[
    {L}(t,x,p)( \alpha_{n+1},\beta_{n+1})=
    \begin{cases}
    V(x)+1 \qquad &\text{ for } \nabla^h V_n = 0,\\
    V(x)+1-\frac{\nabla^h V_n}{|\nabla^h V_n|}\cdot p \qquad &\text{ for } |\nabla^h V_n| \leq 1,\\
        V(x)-1+\frac{\nabla^h V_n}{|\nabla^h V_n|}\cdot p \qquad &\text{ for } |\nabla^h V_n|\geq 1.
    \end{cases}
    \]
    Thus, in this particular example, \eqref{4.1} has an explicit formulation and can be solved numerically rather quickly.
\end{ex}

%-------------------------------------------------------------------------------------------------

\section{Convergence of the discrete equations}
\label{S4}

\quad Let \( V^{\tau,h} \) and \( v \) be the solutions to \eqref{4.2} and \eqref{3.1}, respectively. The goal of this section is to establish a quantitative convergence of \( V^{\tau,h} \) to \( v \) as \( \tau, h \to 0 \).

\subsection{Regularity of continuous equations}
Let us consider the continuous equation with a non-degenerate diffusion: for any $\delta\in (0,1)$ and $\eps\in [0,1)$,
\begin{equation}
\label{9.1}
\begin{cases}
\partial_t v^{\delta,\eps}(t,x) + H(t,x, \nabla v^{\delta,\eps}) =-\frac12\sum_{i=1}^d(\Sigma_i(t,x) +\delta) \partial_{x_i}^2 v^{\delta,\eps} \qquad &\text{ in } (0,T+\eps)\times \mathbb{R}^d,\\
 v^{\delta,\eps}(T+\eps,x) = g(x) \qquad &\text{ on }  \mathbb{R}^d.
 \end{cases}
\end{equation}
Here, we extend $H$ and $\Sigma$ for $t\in [T,T+\eps]$ smoothly if needed. Later, we denote $T_\eps:=T+\eps$.

\quad We note first that $v^{\delta,\eps}$ is uniformly Lipschitz continuous independent of $\delta>0$.

\begin{lemma}\lb{L.9.1}
Assume $g$ to be uniformly $C^2$. There exists $C>0$ independent of $\delta\in (0,1)$ such that
\[
\|\partial_t v^{\delta,\eps}\|_\infty+\|\nabla v^{\delta,\eps}\|_\infty\leq C.
\]
\end{lemma}
\begin{proof}
The proof is standard following the line of Bernstein's method. 
We skip the proof and refer the reader to \cite{LMT,tranbook}.
\end{proof}

\quad Next, we bound the H\"{o}lder norm of $\nabla v^{\delta,\eps}$ in both space and time. 

\begin{lemma}\lb{L.9.3}
%For any $\alpha\in (0,1]$, there exists $C\geq 1$ independent of $\delta$ such that for each $t$,
%\[
%\| \nabla v^{\delta,\eps}(t,\cdot)\|_{C^{\alpha}(\bbR^d)}\leq C\delta^{-(1+\alpha)}.
%\]   
%Furthermore, assuming $g$ to be uniformly $C^3$, then 
%\[
%\|\nabla v^{\delta,\eps}(\cdot,x)\|_{C^{\frac{1+\eta}2}}\leq  C\delta^{-\frac{3+\eta}{2}}
%\]
%uniformly for all $x\in\bbR^d$.

Assuming $g$ to be uniformly $C^3$, for any $\alpha\in (0,1)$, there exists $C>0$ independent of $\delta\in (0,1)$ such that
\[
\|\nabla v^{\delta,\eps}(\cdot,\cdot)\|_{C^{\alpha}((0,T_\eps)\times \bbR^d)}\leq  C \delta^{-(1+\alpha)}.
\]
\end{lemma}
\begin{proof}
Let us fix $(t_0,x_0)$, and we assume $x_0=0$ after shifting. Let $\delta_i:=\Sigma_i(t_0,0)+\delta\geq\delta$, and for $r\in(0,1)$ define
\[%\beq\lb{9.2}
\tilde v(t,x):= v^{\delta,\eps}\left(t_0+r^2t,\delta_1^{1/2}r x_1,\ldots,\delta_d^{1/2}r x_d\right).
\]%\eeq
Then $\tilde v$ satisfies
\beq\lb{9.3}
\partial_t\tilde  v   +\frac12\sum_{i=1}^d \tilde\Sigma_i( t, x) \partial_{x_i}^2\tilde  v=\tilde f(t,x)
\eeq
where
\begin{align*}
\tilde \Sigma_i(t,x)&:=\frac{\Sigma_i\left(t_0+r^2t,\delta_1^{1/2}r x_1,\ldots,\delta_d^{1/2}r x_d\right)+\delta}{\delta_i},
\\
\tilde f(t,x)&:= r^2 H(\cdot,\cdot,\nabla v^{\delta,\eps})\left(t_0+r^2 t,\delta_1^{1/2}r x_1,\ldots,\delta_d^{1/2}r x_d\right).    
\end{align*}
We will consider $\tilde v$ in the domain of $\tilde Q_1$ where for $R>0$,
\beq\lb{9.5}
\tilde Q_R:=(-R,\min\{R,(T_\eps-t_0)/r^2\})\times B_{R}.
\eeq
The corresponding domain for $v^{\delta,\eps}$ is then given by
\[
Q_1^r:=\left\{\left(t_0+r^2t,\delta_1^{1/2}r x_1,\ldots,\delta_d^{1/2}r x_d\right)\,\big|\, (t,x)\in\tilde Q_1\right\}.
\]

\quad Since $\sigma_i={\Sigma_i^{1/2}}$ is assumed to be Lipschitz continuous and $\Sigma_i(t_0,0)=\delta_i-\delta$, there exists $C>0$ such that for any $t$ and $x$ we have
\[
|\Sigma_i(t,x)-\Sigma_i(t_0,0)|\leq C(|t|+|x|)({\delta_i^{1/2}}+|x|+|t|).
\]
Note that for $(t,x)\in Q_1^r$, $|t|\leq r^2$ and $|x|\leq Cr$ as $\delta_i$ are uniformly bounded.
Thus, if we pick $r=\delta^{1/2}/{C_1}$ for some $C_1$ large enough but independent of $\delta_i$ and $\delta$, 
\beq\lb{9.6}
\Sigma_i(\cdot,\cdot)+\delta\in  \left[\frac{\delta_i}2,2\delta_i\right] \text{ in } Q_1^r, \quad \text{and}\quad\tilde\Sigma_i(\cdot,\cdot)\in  \left[1/2,2\right]\text{ in } \tilde Q_1.
\eeq
Moreover, we claim that $\tilde\Sigma_i$ is Lipschitz continuous in $\tilde Q_1$. Indeed, for $(t,x),(t+s,x+y)\in Q_1^r$, since $\sigma_i={\Sigma_i^{1/2}}$ is Lipschitz continuous and by \eqref{9.6}, we have
\[
|\Sigma_i(t+s,x+y)-\Sigma_i(t,x)|\leq C(|s|+|y|)(\sigma_i(t+s,x+y)+\sigma_i(t,x))\leq C{\delta_i^{1/2}}(|s|+|y|).
\]
This implies that for $(t,x),(t+s,x+y)\in \tilde Q_1$
\beq\lb{9.7}
|\tilde\Sigma_i(t+s,x+y)-\tilde\Sigma_i(t,x)|\leq \frac{C{\delta_i^{1/2}}(r^2|s|+r|y|)}{\delta_i}\leq C(|s|+|y|),
\eeq
where we used $r={\delta^{1/2}}/C_1$ in the last inequality.

\quad Next, it is clear that $\tilde f$ is uniformly bounded, and $\tilde v$ is uniformly bounded and Lipschitz continuous. It follows from \cite[Theorem 12.10]{lieberman} that there exist $\eta\in (0,1)$ and $C>0$ such that
\[
\|\nabla \tilde v\|_{C^{\eta}(\tilde Q_{1/2})}\leq  C
\]
where $\tilde Q_{1/2}$ is defined in \eqref{9.5} and the constant is independent of $(t_0,x_0)$ and $\delta$.
After rescaling and using that $(t_0,x_0)$ is arbitrary, we get
\[
\|\nabla v^{\delta,\eps}\|_{C^{\eta}((0,T_\eps)\times\bbR^d)}\leq  C\delta^{-(1+\eta)}.
\]

\quad Having uniform H\"{o}lder continuity of $\nabla\tilde v$ in both space and time in $\tilde Q_{1/2}$, we have that $\tilde f$ is also uniformly H\"{o}lder continuous in both space and time. Moreover, by \eqref{9.6} and \eqref{9.7}, equation \eqref{9.3} is uniformly non-degenerate in $\tilde Q_1$, and $\tilde\Sigma_i$ is uniformly Lipschitz continuous for each $i$. Also, note that the terminal data is $C^3$ since $g$ is $C^3$ and $\tilde\Sigma_i$ is uniformly Lipschitz continuous by \eqref{9.7}. Hence, it follows from Schauder's estimate that there exists $\eta\in (0,1)$ such that
\[
\|\tilde v\|_{C^{1+\eta/2,2+\eta}(\tilde Q_{1/2})}\leq C.
\]
This and an interpolation result from  \cite[Exercise 8.8.6]{krylov1996lectures} show that $\nabla\tilde v$ is uniformly $\frac{1+\eta}{2}$-H\"{o}lder continuous in time. It is clear that $\nabla \tilde v$ is Lipschitz in space with uniformly finite Lipschitz constant. Thus, for any $\alpha\in (0,1)$, we have $\nabla \tilde v\in C^{\alpha/2,\alpha}$ with H\"{o}lder norm independent of $i$, $\delta_i$, $\delta$ and $(t_0,x_0)$. By Schauder estimates once more, $\tilde v\in C^{1+\alpha/2,1+\alpha}$ and is uniformly bounded in the space in $\tilde Q_1$ with H\"{o}lder norm independent of $i$, $\delta_i$, $\delta$ and $(t_0,x_0)$. 

\quad By  \cite[Exercise 8.8.6]{krylov1996lectures} again, $\nabla\tilde v$ is uniformly $\frac{1+\alpha}{2}$-H\"{o}lder continuous in time.
Thus, we get for any $\alpha\in (0,1)$, $\nabla \tilde v\in C^{\alpha,\alpha}(\tilde Q_{1/2})$ with  H\"{o}lder norm independent of $i$, $\delta_i$, $\delta$ and $(t_0,x_0)$. 
After rescaling, for each $\alpha\in (0,1)$ and $\beta\in (0,1]$, we obtain
\[
\|\nabla v^{\delta,\eps}(\cdot,x)\|_{C^\alpha}\leq  C\delta^{-(1+\alpha)}\quad\text{ and }\quad \|\nabla v^{\delta,\eps}(t,\cdot)\|_{C^\beta}\leq  C\delta^{-(1+\beta)}
\]
uniformly for all $x$ and $t$.

\end{proof}

\subsection{Convolution regularization}
We use $v^{\delta,\eps}$ to approximate the solution to the finite-difference scheme. We first use convolution to regularize $v^{\delta,\eps}$ with regularization parameter $\eps\in (0,1)$ and then we will optimize over $\eps$.

\quad Take a non-negative function $\zeta\in C_0^{\infty}(\bbR^{d+1})$ with support in $(-1,1)\times B_1$ and with unit integral. For any $\eps\in (0,1)$, define
\[
\zeta_\eps(t,x):=\eps^{-d-1}\zeta(t/\eps,x/\eps)
\]
and for any smooth function $u:\bbR^{d+1}\to \bbR$, define
\[
u_\eps(t,x):= \int_{\bbR^{d+1}}u(s,y)\zeta_\eps(t-s,x-y)\,dsdy.
\]
This is distinct from \cite{18Kry}, where a parabolic scale is applied. The difference is due to the fact that our solutions are Lipschitz continuous in both space and time, while the solutions in \cite{18Kry} are Lipschitz continuous in space and $\frac12$-H\"{o}lder continuous in time.

\quad Next, we estimate the difference between finite differences and derivatives for smooth functions. 
Again, let $u:\bbR^{d+1}\to \bbR$ be smooth. It follows from Taylor's formula that we have for some dimensional constant $C>0$,
\beq\lb{8.1}
|\partial_tu(t,x)-\partial_t^\tau u(t,x)|\leq C\tau \|\partial_t^2 u\|_\infty,
\eeq
for any $i=1,\ldots,d$,
\[%\beq\lb{8.2}
|\nabla u(t,x)-\nabla^h u(t,x)|+ |\partial_{x_i} u(t,x)-D_i^hu(t,x)| \leq Ch \|D^2_x u\|_\infty,
\]%\eeq
and
\beq\lb{8.3}
|\partial_{x_i}^2 u(t,x)-\Delta_i^h u(t,x)|\leq Ch^2 \|D_x^4 u\|_\infty.
\eeq
In the last estimate, we applied the particular form of discrete second order derivatives.

\medskip

\quad Let $1>\delta\gg\eps\gg h>\tau>0$. Define
\beq\lb{conv}
u^{\delta}_{\eps}(t,x):=\int_{\bbR^{d+1}}v^{\delta,\eps}(s,y)\zeta_\eps(t-s,x-y)\,dsdy.
\eeq
Since $v^{\delta,\eps}$ is defined for $t\leq T+\eps$, $u^{\delta}_{\eps}$ is well-defined for all $(t,x)\in (0,T]\times\bbR^d$. 
Making use of the Lipschitz continuity of $v^{\delta,\eps}$, we immediately have the following properties.

\begin{lemma}\lb{L.9.5}
In the domain of $(t,x)\in (0,T]\times\bbR^d$,
\[
 \|\partial_t^2 u^{\delta}_{\eps}\|_\infty,\, \|D^2_x u^{\delta}_{\eps}\|_\infty
\leq C \eps^{-1},\quad
\|D_x^4 u^{\delta}_{\eps}\|_\infty\leq C\eps^{-3}.
\]
Consequently, we have
\[
|\partial_tu^{\delta}_{\eps}(t,x)-\partial_t^\tau u^{\delta}_{\eps}(t,x)|\leq C\tau \eps^{-1},
\]
and for any $i=1,\ldots,d$,
\[
%|\nabla u^{\delta}_{\eps}(t,x)-\nabla^h u^{\delta}_{\eps}(t,x)|,\quad 
|\partial_{x_i} u^{\delta}_{\eps}(t,x)-D_i^hu^{\delta}_{\eps}(t,x)| \leq Ch \eps^{-1},
\]
\[
|\partial_{x_i}^2 u^{\delta}_{\eps}(t,x)-\Delta_i^h u^{\delta}_{\eps}(t,x)|\leq Ch^2 \eps^{-3}.
\]
\end{lemma}
\begin{proof}
The first three inequalities follow from the uniform space-time Lipschitz continuity of $v^{\delta,\eps}$ and the properties of convolution. For instance,
\begin{align*}
\left|\partial_{x_i}^4 u^{\delta}_{\eps}(t,x)\right|&=\left|\int_{\bbR^{d+1}}\partial_{x_i}v^{\delta,\eps}(t-s,x-y)\partial_{x_i}^3\zeta_\eps(s,y)\,dyds\right|\\
&\leq C\int_{\bbR^{d+1}}\left|\partial_{x_i}^3\zeta_\eps(s,y)\right|\,dyds\leq C\eps^{-3}.
\end{align*}
The rest of the claims follow from \eqref{8.1}--\eqref{8.3}.

\end{proof}

\begin{proposition}\lb{P.9.6}
For any $\alpha\in(0,1)$, there exists $C>0$ such that
\[
\left|\partial_t u^{\delta}_{\eps}(t,x) + H(t,x, \nabla u^{\delta}_{\eps})+\frac12\sum_{i=1}^d(\Sigma_i(t,x) +\delta) \partial_{x_i}^2 u^{\delta}_\eps(t,x)\right|\leq C\delta^{-(1+\alpha)}\eps^\alpha,
\]
and
\begin{multline*}
\left|\partial^\tau_t u^{\delta}_{\eps}(t,x) + H(t,x, \nabla^h u^{\delta}_{\eps})+\frac12\sum_{i=1}^d(\Sigma_i(t,x) +\delta) \Delta_i^h u^{\delta}_\eps(t,x)\right|\\
\leq C\left(\tau\eps^{-1}+ h^2\eps^{-3}+h\eps^{-1}+\delta^{-(1+\alpha)}\eps^\alpha\right).
\end{multline*}
\end{proposition}
\begin{proof}
We convolute $\zeta_\eps$ to the equation \eqref{9.1} of $v^{\delta,\eps}$ to get   
\begin{align}\lb{9.8}
    \partial_t u^{\delta}_{\eps}(t,x) + H(t,x, \nabla u^{\delta}_{\eps})+\frac12\sum_{i=1}^d(\Sigma_i(t,x) +\delta) \partial_{x_i}^2 u^{\delta}_\eps(t,x)=X_1+X_2
\end{align}
where
\[
X_1:=\int_{\bbR^{d+1}}\left[H(t,x, \nabla u^{\delta}_{\eps})-H(\cdot,\cdot, \nabla v^{\delta,\eps})(t-s,x-y)\right]\zeta_\eps(s,y)\,dyds,
\]
\[
X_2:=\frac12\int_{\bbR^{d+1}} \sum_{i=1}^d\left[\Sigma_i(t,x) -\Sigma_i(t-s,x-y)\right] \partial_{x_i}^2 v^{\delta,\eps}(t-s,x-y) \zeta_\eps(s,y)\,dyds.
\]
Here, $X_1, X_2$ are commutation errors from the convolution with the standard kernel $\zeta_\eps$.
By Lemma \ref{L.9.3} and \eqref{conv}, for $(s,y)$ in a space-time ball of radius $\eps$ and center $(0,0)$, and for any $\alpha\in (0,1)$, there exists $C>0$ such that
\begin{align*}
&\quad\, \left|\nabla u_\eps^\delta(t,x)-\nabla v^{\delta,\eps}(t-s,x-y)\right|\\
&=\left|\int_{\bbR^{d+1}}\nabla v^{\delta,\eps}(t-s',x-y')\zeta_\eps(s,y)\,ds'dy'-\nabla v^{\delta,\eps}(t-s,x-y)\right|\leq C\delta^{-(1+\alpha)}\eps^\alpha.  
\end{align*}
Then using that $H(\cdot,\cdot,\cdot)$ is Lipschitz with respect to its variables and the assumption that $\zeta_\eps$ is supported in a space-time ball of radius $\eps$, we obtain
\[
|X_1|\leq \int_{\bbR^{d+1}}(C\eps+C\delta^{-(1+\alpha)}\eps^\alpha)\zeta_\eps(s,y)\,dyds\leq C\delta^{-(1+\alpha)}\eps^\alpha.
\]

\quad As for $X_2$, we have
\begin{align*}
&|X_2|\\
=\,&\frac12\left|\int_{\bbR^{d+1}} \sum_{i=1}^d\left[\Sigma_i(t,x) -\Sigma_i(t-s,x-y)\right] \partial_{y_i}\partial_{x_i} v^{\delta,\eps}(t-s,x-y) \zeta_\eps(s,y)\,dyds\right|\\
=\,&\frac12\left|\int_{\bbR^{d+1}} \sum_{i=1}^d\partial_{y_i}\left[(\Sigma_i(t,x) -\Sigma_i(t-s,x-y))\zeta_\eps(s,y)\right] \partial_{x_i} v^{\delta,\eps}(t-s,x-y) \,dyds\right|\\
=\,&\frac12\left|\int_{\bbR^{d+1}} \sum_{i=1}^d\partial_{y_i}\left[(\Sigma_i(t,x) -\Sigma_i(t-s,x-y))\zeta_\eps(s,y)\right] \partial_{x_i} \left[v^{\delta,\eps}(t-s,x-y) - v^{\delta,\eps}(t,x) \right]\,dyds\right|.    
\end{align*}
Note that for $(s,y)$ as the above,
\begin{align*}
&\quad\,\left|\partial_{y_i}\left[(\Sigma_i(t,x) -\Sigma_i(t-s,x-y))\zeta_\eps(s,y)\right]\right|\\
 & \leq    \left|\left[\Sigma_i(t,x) -\Sigma_i(t-s,x-y)\right]\partial_{y_i}\zeta_\eps(s,y)\right|+\left|\partial_{x_i}\Sigma_i(t-s,x-y)\zeta_\eps(s,y)\right|\\
 &\leq C\eps |\nabla \zeta_\eps(s,y)|+C\zeta_\eps(s,y),
\end{align*}
and, by Lemma \ref{L.9.3},
\[
\left|\partial_{x_i} v^{\delta,\eps}(t-s,x-y) -\partial_{x_i} v^{\delta,\eps}(t,x) \right|\leq C\delta^{-(1+\alpha)}\eps^\alpha.
\]
Hence
\[
|X_2|\leq C\delta^{-(1+\alpha)}\eps^\alpha\left(1+\eps\int_{\bbR^{d+1}}|\nabla \zeta_\eps(s,y)|\,dyds\right)\leq C\delta^{-(1+\alpha)}\eps^\alpha.
\]
Thus, we conclude the first claim from \eqref{9.8}.

\quad It follows from Lemma \ref{L.9.5} that
\[
\left|\partial_t u^{\delta}_{\eps}-\partial^\tau_t u^{\delta}_{\eps}\right|\leq C\tau\eps^{-1},\quad \left|\partial_{x_i}^2 u^{\delta}_\eps-\Delta_i^h u^{\delta}_\eps\right|\leq Ch^2\eps^{-3},
\]
and, also using the Lipscthiz regularity assumption on $H$,
\[
\left|H(t,x, \nabla u^{\delta}_{\eps})- H(t,x, \nabla^h u^{\delta}_{\eps})\right|\leq C|\nabla u^{\delta}_{\eps}-\nabla^h u^{\delta}_{\eps}|\leq Ch\eps^{-1}.
\]
Hence, these and the first claim yield the second claim.
\end{proof}

\quad As a corollary of the proposition, we can estimate the difference between $u^{\delta,\eps}$  and $V^{\tau,h}$.

\begin{corollary}\lb{C.4.5}
For any $\alpha\in(0,1)$, there exists $C>0$ such that
\[
\left|u^{\delta}_{\eps}-V^{\tau,h}\right|\leq C\nu_h^{-(1+\alpha)}h^{\alpha/2}.
\]
Actually, the right-hand side can be improved to
\[
C\min_{\eps\in (0,1]}\left( h^2\eps^{-3}+h\eps^{-1}+\nu_h^{-(1+\alpha)}\eps^\alpha\right).
\]
\end{corollary}

\begin{proof} 
Let us take $\delta:=\nu_h$. Denote 
\[
\eps_{\delta,h,\tau}:=C\left(\tau\eps^{-1}+ h^2\eps^{-3}+h\eps^{-1}+\delta^{-(1+\alpha)}\eps^\alpha\right)
\]
from Proposition \ref{P.9.6}, and for some $C_0>0$, then set
\[
u_1(t,x):=u_\eps^{\delta}(t,x)+\eps_{\delta,h,\tau}(T-t)+C_0\eps,
\]
and
\[
u_2(t,x):=u_\eps^{\delta}(t,x)-\eps_{\delta,h,\tau}(T-t)-C_0\eps.
\]
Then, since $\delta=\nu_h$, it follows from the proposition that $u_1$ and $u_2$ are, respectively, super- and sub- solutions to \eqref{4.1} in the domain of $(0,T]$. By Lemma \ref{L.9.1} and \eqref{conv}, $v^{\delta,\eps}$ and then $u_\eps^{\delta}$ is Lipschitz continuous with Lipschitz uniform constant and $v^{\delta,\eps}(T+\eps,x)=g(x)$. Therefore, we can pick $C_0$  sufficiently large such that
\[
u_2(T,x)\leq V^{\tau,h}(T,x)\leq u_1(T,x).
\]
It follows from the monotonicity property of the discrete scheme,
\[
u_\eps^{\delta}(t,x)-C\eps_{\delta,h,\tau}\leq u_2(t,x)\leq V^{\tau,h}(t,x)\leq u_1(t,x)\leq u_\eps^{\delta}(t,x)+C\eps_{\delta,h,\tau}.
\]
The  conclusion follows after taking $\eps:=h^{1/2}$.
\end{proof}

\begin{remark}\lb{R.2}
If the diffusion term is non-degenerate, that is, for some $c_0\in (0,1)$ we have $1/c_0\geq \Sigma_i\geq c_0$ for all $i=1,\ldots,d$, then we take $\nu_h=0$ and \eqref{N2} still holds when $\tau\ll h^2\ll 1$. We re-define $v^{c_0,\eps}$ to be the solution to \eqref{9.1} with $\delta=c_0$ and with $\Sigma_i-c_0$ in place of $\Sigma_i$. 
For $u_\eps^{c_0}$ defined in \eqref{conv} with the above $v^{c_0,\eps}$, the result of Proposition \ref{P.9.6} holds the same with $c_0$ in place of $\delta$. It follows from the proof of Corollary \ref{C.4.5} that, for $V^{\tau,h}$ solving \eqref{4.2} with $\nu_h=0$, we have
\[
\left|u^{\delta}_{\eps}-V^{\tau,h}\right|\leq C\min_{\eps\in (0,1]}\left( h^2\eps^{-3}+h\eps^{-1}+\eps^\alpha\right)\leq Ch^{\alpha/2}
\]
with $C$ independent of $h$.
\end{remark}

\subsection{The quantitative convergence result}
In this subsection, we establish the convergence of the discrete equations for both cases: whether 
$\Sigma$ is uniformly elliptic or when it is degenerate.
Let us mention that it was proved in \cite{10DonKry,DonKry,18Kry} that 
\[
\sup_{(t,x)\in \Omega^{\tau,h}_T} |V^{\tau,h}(t,x)-v(t,x)|\leq C(\tau^{1/4}+h^{1/2})\quad\text{ for some $C=C(T)>0$},
\]
where $v$ solves a degenerate parabolic Bellman (convex) equation and $V^{\tau,h}$ is its space-time finite difference approximation obtained using an implicit scheme. Convexity is essentially needed in the papers.

\quad We now proceed to prove Theorem \ref{T.4.3}.
In fact, we will prove the following theorem, which implies 
Theorem \ref{T.4.3} right away.
\begin{theorem}%\lb{T.4.3-n}
Assume \eqref{N2}. 
Assume further that $c(t,x,\alpha,\beta),\,f(t,x,\alpha,\beta),\,\sigma(t,x)$ are uniformly Lipschitz in $(t,x)$, and $\|g\|_{C^3}<\infty$.
Let \( V^{\tau,h} \) and \( v \) be the solutions to \eqref{4.2} and \eqref{3.1}, respectively. 
For any $\alpha\in(0,1)$, there exists $C>0$ depending only on $d$, $\|g\|_{C^3}$, {$\|c\|_{\rm Lip}$, $\|f\|_{\rm Lip}$,  and $\|\sigma\|_{\rm Lip}$} such that
\begin{align}\label{eq:rate-discrete}
\sup_{(t,x)\in \Omega^{\tau,h}_T} |V^{\tau,h}(t,x)-v(t,x)|&\leq C\min_{\eps\in (0,1]}\left( h^2\eps^{-3}+h\eps^{-1}+\nu_h^{-(1+\alpha)}\eps^\alpha\right)+C\nu_h^{1/2}\notag\\
&\leq C\nu_h^{-(1+\alpha)}h^{\alpha/2}+C\nu_h^{1/2}.
\end{align}
If the diffusion term is non-degenerate, that is, $1/c_0\geq \Sigma_i\geq c_0$ for some $c_0\in(0,1)$ independent of $h$ and $\tau$ and for all $i=1,\ldots, d$, then we can set $\nu_h=0$ and we obtain
\[
\sup_{(t,x)\in \Omega^{\tau,h}_T} |V^{\tau,h}(t,x)-v(t,x)|\leq
Ch^{\alpha/2}.
\]
In the case when $\Sigma$ is degenerate, that is, $\min_{1\leq i\leq d} \min_{(t,x)} \Sigma_i(t,x)=0$, by choosing $\eps=\nu_h^{(3+2\alpha)/({2\alpha})}$ and $\nu_h=h^{{4\alpha}/({9+7\alpha})}$ in the first inequality of \eqref{eq:rate-discrete}, we have
\[
\sup_{(t,x)\in \Omega^{\tau,h}_T} |V^{\tau,h}(t,x)-v(t,x)|\leq
Ch^{2\alpha/({9+7\alpha})}.
\]
\end{theorem}

%$\nu_h=h^{\alpha/(3+2\alpha)}$, we have
%\[
%\sup_{(t,x)\in \Omega^{\tau,h}_T} |V^{\tau,h}(t,x)-v(t,x)|\leq
%Ch^{\alpha/(6+4\alpha)}.
%\]
%This convergence rate is close to $h^{1/10}$ as $\alpha \to 1$.
\begin{proof}
In view of \eqref{conv} and Lipschitz continuity of $v^{\delta,\eps}$,
\beq\lb{9.10}
|v^{\delta,\eps}-u^{\delta}_{\eps}|\leq C\eps.
\eeq
Take $\delta=\nu_h$. Then, by the triangle inequality,  the first conclusion follows from Corollary \ref{C.4.5} and Proposition \ref{P.5.4} below, after taking $\eps=h^{1/2}$.

\smallskip

\quad If $\Sigma$ is non-degenerate, then $1/c_0\geq \Sigma_i\geq c_0$ for some $c_0\in (0,1)$ and all $i=1,\ldots,d$. We take $\delta=c_0$ and replace $\Sigma_i$ by $\Sigma_i-\delta$. The equation \eqref{9.1} becomes
\begin{equation*}
\begin{cases}
\partial_t v^{\delta,\eps}(t,x) + H(t,x, \nabla v^{\delta,\eps}) =-\frac12\sum_{i=1}^d((\Sigma_i(t,x) -\delta)+\delta) \partial_{x_i}^2 v^{\delta,\eps} \, &\text{ in } (0,T+\eps)\times \mathbb{R}^d,\\
 v^{\delta,\eps}(T+\eps,x) = g(x) \, &\text{ on }  \mathbb{R}^d.
 \end{cases}
\end{equation*}
Then $v^{\delta,0}$ is the same as $v$. Since $v^{\delta,\eps}$ is uniformly Lipschitz continuous, by comparing $v^{\delta,0}\pm C\eps$ with $v^{\delta,\eps}$, we get
%Taking $\delta=c$ and replacing $\Sigma_i$ by $\Sigma_i-\delta+\nu_h$, we re-define $v^{\delta,\eps}$ as the solution to 
%It is direct to see that $v^{\delta,0}=v$. Since $v^{\delta,\eps}$ is uniformly Lipschitz continuous, by comparing $v^{\delta,0}\pm C\eps$ with $v^{\delta,\eps}$, we get
\[
|v^{\delta,0}-v^{\delta,\eps}|\leq C\eps.
\]
Note that in the non-degenerate case,  \eqref{N2} holds when $\tau\ll h^2\ll 1$. Thus, we let $V^{\tau,h}$ solve \eqref{4.2} with $\nu_h=0$.
Then, after taking $\eps=h^{1/2}$, Corollary \ref{C.4.5} (see also Remark \ref{R.2}) and \eqref{9.10} yield
\[
\left|v-V^{\tau,h}\right|\leq \left|v^{\delta,\eps}-V^{\tau,h}\right|+C\eps\leq \left|u^{\delta}_{\eps}-V^{\tau,h}\right|+C\eps\leq Ch^{\alpha/2}.
\]  
\end{proof}

\quad In the following  proposition, we apply the classical viscosity solution approach to estimate the difference between $v^{\delta,\eps}$ and $v$.

\begin{proposition}\lb{P.5.4}
There exists $C>0$ such that for $(t,x)\in (0,T]\times\bbR^d$,
\[
|v(t,x)-v^{\delta,\eps}(t,x)|\leq C\left(\delta^{1/2}+\eps \right).
\]
\end{proposition}

\begin{proof}
Let us only prove the estimate for $v-v^{\delta,\eps}$, and the one for $v^{\delta,\eps}-v$ is almost identical.
Since $v^{\delta,\eps}$ is uniformly Lipschitz continuous in dependent of $\delta$ by Lemma \ref{L.9.1} and $v^{\delta,\eps}(T+\eps,x)=g(x)$,  
we have $v(T,x)-(v^{\delta,\eps}(T,x)+C\eps)\leq 0$ for some $C$ sufficiently large. Note that $v^{\delta,\eps}+C\eps$ satisfies the same equation as $v^{\delta,\eps}$ does. Therefore, after replacing $v^{\delta,\eps}$ by $v^{\delta,\eps}+C\eps$,  it suffices to prove 
\[
v-v^{\delta,\eps}\leq C\delta^{1/2}
\]
under the assumption that $v(T,x)-v^{\delta,\eps}(T,x)\leq 0$.

\quad Let $3\gamma:=\sup_{(t,x)\in [0,T]\times\bbR^d} ({v(t,x)-v^{\delta,\eps}(t,x)})$ and assume that $\gamma>0$, otherwise there is nothing to prove. Then let $R>0$ be sufficiently large such that
\beq\lb{3334}
\sup_{(t,x)\in [0,T]\times B_R} [{v(t,x)-v^{\delta,\eps}(t,x)}]\geq 2\gamma.
\eeq
Let $R_1:=AR$ for some $A\geq 2$ sufficiently large, and we consider a radially symmetric, and radially non-decreasing function $\phi:\bbR^d\to [0,\infty)$ such that 
\beq\lb{3336}
\begin{cases}
\phi(x)\equiv 0 \quad &\text{ for $x\in B_R$},\\
\phi(x)\geq \|v^{\delta,\eps}\|_\infty+\|v\|_\infty \quad &\text{ for } x\in \bbR^n \setminus B_{R_1},
\end{cases}
\eeq
and for some $C>0$,
\beq\lb{3335}
\begin{cases}
|\phi(x)|\leq C \quad &\text{ for } x\in\bbR^n,\\
|\nabla\phi(x)|+ |\nabla^2\phi(x)|\leq C/A\quad &\text{ for } x\in \bbR^n.
\end{cases}
\eeq
Below, all constants' dependence on $A$ will be explicit, and all $C$'s are independent of $A$.

\quad Due to \eqref{3334}, \eqref{3336} and the assumption that $v(T,\cdot)\leq v^{\delta,\eps}(T,\cdot)$, there exists $(t_0,x_0)\in [0,T)\times B_{R_1}$ such that
\beq\lb{3.71}
\begin{aligned}
&v(t_0,x_0)-v^{\delta,\eps}(t_0,x_0)-\frac{T-t_0}{T}\gamma-2\phi(x_0)\\
=\, &\sup_{(t,x)\in[0,T]\times \bbR^d} \left[v(t,x)-v^{\delta,\eps}(t,x)-\frac{T-t}{T}\gamma-2\phi(x)\right]=:\gamma'\geq \gamma.  
\end{aligned}
\eeq
We write
\[
v^\gamma(t,x):=v^{\delta,\eps}(t,x)+\frac{T-t}{T}\gamma.
\]
For any $\beta\geq 1$, there are $(t_1,x_1),(t_2,x_2)\in [0,T)\times B_{R_1}$ such that
\beq\lb{3.7}
\begin{aligned}
& v(t_2,x_2)-v^\gamma(t_1,x_1)  -\phi(x_1)-\phi(x_2)- \beta \left(|x_1 - x_2|^2+|t_1 - t_2|^2\right)\\
 =\,& \sup_{
(t,x),(t',y')\in [0,T]\times \bbR^d}
\left[ v(t,x) - v^\gamma(t',x') - \phi(x)-\phi(x')-\beta \left(|x - x'|^2+|t-t'|^2\right)\right]\\
\geq\,& v(t_0,x_0)-v^\gamma(t_0,x_0)-2\phi(x_0)=\gamma'.
\end{aligned}
\eeq
Since $v$ and $\phi$ are Lipschitz continuous, it follows from \eqref{3.71} and \eqref{3.7} that
\begin{align*}
\gamma' &\leq   v(t_1,x_1)-v^\gamma(t_1,x_1) -2\phi(x_1)+C(|t_1-t_2|+|x_1-x_2|)-\beta \left(|x_1 - x_2|^2+|t_1 - t_2|^2\right)\\
&\leq \gamma'+C|t_1-t_2|+C|x_1-x_2|-\beta \left(|x_1 - x_2|^2+|t_1 - t_2|^2\right),
\end{align*}
which implies that
\beq\lb{3.8}
|t_1-t_2|+|x_1-x_2|\leq \frac{C}{\beta}.
\eeq

\quad  
Note that $v^\gamma$ satisfies
\[
    \partial_t v^\gamma(t,x)+\frac{\gamma}T +H(t,x, \nabla v^{\gamma}(t,x)) +\frac12\sum_{i=1}^d(\Sigma_i(t,x) +\delta) \partial_{x_i}^2 v^{\gamma}= 0
\]
in the viscosity sense, and $v$ is the solution to \eqref{3.1}.
Thus, the Crandall-Ishii lemma \cite[Theorem 8.2]{user} yields that there exist two symmetric matrices $X_1$, $X_2$ satisfying the following:
\beq\lb{3.111}
-(2\beta+|J|)I\leq
\begin{pmatrix}
X_2 & 0\\
0 & -X_1
\end{pmatrix}
\leq J+\frac{1}{2\beta}J^2,
\quad\text{ with }
J:=2\beta\begin{pmatrix}
I & -I\\
-I & I
\end{pmatrix},
\eeq
and
\begin{multline}\lb{3.22}
\frac{\gamma}T+    H(t_1,x_1, p_1) +\frac12 \Tr 
\left[(\Sigma(t_1,x_1)+\delta) (X_1+D^2\phi(x_1))\right]\leq 0\\
   \leq H(t_2,x_2, p_2) +\frac12 \Tr \left[\Sigma(t_2,x_2) (X_2-D^2\phi(x_2))\right],
\end{multline}
and
\[
p_1:=2\beta (x_1-x_2)+\nabla\phi(x_1),\quad p_2:=2\beta (x_1-x_2)-\nabla\phi(x_2).
\]

\quad Now, 
by \eqref{3335} and \eqref{3.8}, we have
\begin{align*}
\begin{cases}
|p_1|+|p_2|\leq C, \\
|p_1-p_2|\leq C/A.
\end{cases}
\end{align*}
By \eqref{3335}, \eqref{3.8} again and Lemma \ref{L.2.3}, \eqref{3.22} can be simplified to
\beq\lb{3.13}
\begin{aligned}
 \frac{\gamma}T&\leq  C(|t_1-t_2|+|x_1-x_2|)(1+|p_1|)+C|p_1-p_2|+C(|D^2\phi(x_1)|+|D^2\phi(x_2)|)\\
 &\qquad+\frac12\Tr 
    (\Sigma( t_2, x_2) X_2-\Sigma( t_1, x_1) X_1)\\
&\leq  C(\beta^{-1}+A^{-1})+CA^{-1}+\frac12\Tr 
    (\Sigma(t_2,x_2) X_2-\Sigma(t_1,x_1) X_1)-\frac{\delta}{2}\Tr X_1.
\end{aligned}
\eeq

\quad It follows from \eqref{3.111} that 
\[
|X_1|,\,|X_2|\leq C\beta.
\]
Note that $\Sigma=\sigma\sigma^T$, $\Sigma$ and $\sigma$ are diagonal matrices, and $\sigma$ is Lipschitz continuous. We multiply \eqref{3.111} by the nonnegative symmetric matrix
\[
\begin{pmatrix}
\sigma(t_2,x_2)\sigma(t_2,x_2)^T & \sigma(t_1,x_1)\sigma(t_2,x_2)^T\\
\sigma(t_2,x_2)\sigma(t_1,x_1)^T & \sigma(t_1,x_1)\sigma(t_1,x_1)^T
\end{pmatrix}
\]
on the left-hand side,
and take traces to obtain
\[
\begin{aligned}
\frac12\Tr(\Sigma(t_2,x_2)X_2 )&-\frac12\Tr( \Sigma(t_1,x_1)X_1 )\leq 3\beta \Tr\left[(\sigma(t_2,x_2)- \sigma(t_1,x_1) )(\sigma(t_2,x_2)- \sigma(t_1,x_1) )^T\right]\\
&\leq C\beta (|t_1-t_2|^2+|x_1-x_2|^2)\leq C/\beta,
\end{aligned}
\]
where in the last inequality, we applied \eqref{3.8}.
Using these in \eqref{3.13} yields
\[
\gamma\leq C(\beta^{-1}+A^{-1})+C\delta\beta+C/\beta.
\]
Note that the inequality holds uniformly for all $A$.
Thus, taking $\beta:={\delta}^{-1/2}$ and passing $A\to \infty$, we obtain
$
\gamma\leq C{\delta}^{1/2}$,
which finishes the proof.
\end{proof}
We refer the reader to \cite[Proposition 2.5]{CGMT} for a different proof of the above proposition.
Further, the convergence rate $O\left(\delta^{1/2}+\eps\right)$ is optimal (see \cite{QSTY}).

\quad As a corollary of Theorem \ref{T.4.3} and Proposition \ref{P.5.4} with $\eps=h^{1/2}$, since $v^{\delta,\eps}$ is uniformly Lipschitz continuous independent of $\delta$, we also derive a regularity result of the discrete solutions.
\begin{corollary}
Under the assumptions of Theorem \ref{T.4.3}, for any $(t,x),(s,y)\in \Omega_T^{\tau,h}$,
\[
|V^{\tau,h}(t,x)-V^{\tau,h}(s,y)|\leq C\left(|t-s|+|x-y|+\nu_h^{-(1+\alpha)}h^{\alpha/2}+\nu_h^{1/2}\right).
\]
If $\Sigma$ is uniformly elliptic, then
\[
|V^{\tau,h}(t,x)-V^{\tau,h}(s,y)|\leq C\left(|t-s|+|x-y|+h^{\alpha/2}\right).
\]
\end{corollary}

\bibliographystyle{abbrv}
%\bibliography{bib}

\end{document}